\documentclass[11pt]{article}
\usepackage[pass]{geometry}
\PassOptionsToPackage{table,dvipsnames}{xcolor}

\usepackage{pdfpages}
\usepackage{natbib}
\usepackage[titletoc,toc,title]{appendix}
\usepackage{csquotes}
\MakeOuterQuote{"}
\usepackage{amsmath, amssymb}
\usepackage{algorithm, algpseudocode, verbatim, multirow, paralist, nicefrac, hyperref}
\usepackage{multicol}
\usepackage{caption, subcaption, array, graphicx, pgfplots, tikz, xcolor, setspace}
\usepackage[section]{placeins}
\usepackage{mathtools}

\usepackage{latexsym}
\usepackage{natbib}
\usepackage{amsthm}
\usepackage{cases}
\usepackage{epsfig}

\usepackage{booktabs}
\usepackage{multirow}
\usepackage{amssymb}   

\usepackage{soul}
\usepackage[table]{xcolor}
\definecolor{omittedgray}{gray}{0.93}
\usepackage{scalefnt}
\usepackage{booktabs}
\usepackage[english]{babel}
\usepackage{graphicx}
\usepackage{titlesec}
\usepackage{epsfig}
\usepackage{enumerate}
\usepackage{caption}
\usepackage{booktabs}
\usepackage[english]{babel}
\usepackage{graphicx}
\usepackage{multirow}
\usepackage{rotating}
\usepackage{fancyhdr}
\usepackage{setspace}
\usepackage[section]{placeins}
\usepackage{subcaption}

\usepackage{breqn}

\usepackage{threeparttable}

\usepackage{array, makecell}
\usepackage{pseudocode}
\usepackage{algorithmicx}
\usepackage{algorithm}
\usepackage{url} 
\usepackage{appendix}
\usepackage{ragged2e}
\usepackage{physics}
\usepackage{amsthm}
\usepackage{chngpage}
\newtheorem{proposition}{Proposition}

\newtheorem{lemma}{Lemma}

\usepackage[nocomma]{optidef}
\usepackage{mathtools,zref-savepos}
\newtagform{roman}[]()

\usepackage{xargs}                      
\usepackage[pdftex,dvipsnames]{xcolor}  
\usepackage[colorinlistoftodos,prependcaption,textsize=small]{todonotes}
\pgfplotsset{compat=1.18}
\newcommandx{\unsure}[2][1=]{\todo[linecolor=red,backgroundcolor=red!25,bordercolor=red,#1]{#2}}

\newcommandx{\change}[2][1=]{\todo[linecolor=blue,backgroundcolor=blue!25,bordercolor=blue,#1]{#2}}

\newcommandx{\info}[2][1=]{\todo[linecolor=OliveGreen,backgroundcolor=OliveGreen!25,bordercolor=OliveGreen,#1]{#2}}

\newcommandx{\improvement}[2][1=]{\todo[linecolor=Plum,backgroundcolor=Plum!25,bordercolor=Plum,#1]{#2}}

\newcommandx{\thiswillnotshow}[2][1=]{\todo[disable,#1]{#2}}

\usepackage{authblk} 
    \usepackage{enumitem}
    \usepackage{hyperref}
	\usepackage{xcolor}
    \usepackage{longtable} 

\makeatletter
\renewcommand\section{\@startsection {section}{1}{\z@}%
                                   {-3.5ex \@plus -1ex \@minus -.2ex}%
                                   {2.3ex \@plus.2ex}%
                                   {\normalfont\fontfamily{phv}\fontsize{16}{19}\bfseries}}
\renewcommand\subsection{\@startsection{subsection}{2}{\z@}%
                                     {-3.25ex\@plus -1ex \@minus -.2ex}%
                                     {1.5ex \@plus .2ex}%
                                     {\normalfont\fontfamily{phv}\fontsize{14}{17}\bfseries}}
\renewcommand\subsubsection{\@startsection{subsubsection}{3}{\z@}%
                                    {-3.25ex\@plus -1ex \@minus -.2ex}%
                                     {1.5ex \@plus .2ex}%
                                     {\normalfont\normalsize\fontfamily{phv}\fontsize{14}{17}\selectfont}}
\makeatother

\newcommand{\blind}{1}	
	
\begin{document}
		
		\def\spacingset#1{\renewcommand{\baselinestretch}%
			{#1}\small\normalsize} \spacingset{1}
		
		\if1\blind
		{
			\title
            {Supply Chain Information Sharing among Competitors: Profit-Based Participation under Market Shocks}
			\author{Batuhan \c{C}elik\thanks{celikb@rpi.edu}
            }
            \author{Bahar \c{C}avdar\thanks{cavdab2@rpi.edu}}
            \author{Jennifer Pazour\thanks{pazouj@rpi.edu}
            }
            \affil{\small{Department of Industrial and Systems Engineering, Rensselaer Polytechnic Institute, Troy, NY}}
			\date{ }
			\maketitle
		} \fi
		
		\if0\blind
{

    \title{\bf 
    Supply Chain Information Sharing among Competitors: Profit-Based Participation under Market Shocks
    }
    \author{Author information is purposely removed for double-blind review}
    
\bigskip
    \bigskip
    \bigskip
    \begin{center}
        {\LARGE\bf 
        Supply Chain Information Sharing among Competitors: Profit-Based Participation under Market Shocks
        }
    \end{center}
    \medskip
} \fi
		\bigskip		

		\vspace{-2cm}
	\begin{abstract}

Sharing information can help coordinate supply chain decisions in response to market shocks. 
Yet, firms must weigh these benefits against the risk of strengthening competitors.
We study when competing firms voluntarily join an information-sharing platform to exchange information about market shocks while retaining operational autonomy. Firms differ in information capabilities and decide whether to participate before private signals are realized, creating stochastic information asymmetry. We model participation and quantity decisions as a two-stage dynamic game with incomplete information. For two horizontal competitors, we derive closed-form equilibrium participation strategies that reveal how information power and production efficiency shape participation. Numerical analysis examines more complex supply chain networks.  We find that supply chain characteristics impact participation and the level of information asymmetry is a key factor in the equilibrium platform profiles. Firms with similar information power may benefit from joining forces to reduce their competitive disadvantage against higher information power competitors. Cost-inefficient firms may be less willing to participate because their competitors can exploit shared information, and uncertainty can serve as a competitive shield by restricting the actions of more efficient competitors.  However, when several competitors join the platform, a cost-inefficient firm may also choose to participate to avoid an informational disadvantage.
	\end{abstract}
			
	\noindent%

	\spacingset{1.5} 

\section{Introduction} \label{s:introduction}

Supply chain firms' ability to adapt their operations in response to changes in market conditions is key to improving supply chain performance, especially when disruptions or market fluctuations cause demand shocks \citep{Eckstein19052015}.
Timely access to supply chain information determines this adaptability \citep{topan2018,zheng2019reactive}. 
However, in most supply chains, information flows indirectly through orders or shipments, requiring firms to infer market conditions retrospectively rather than observe them directly \citep{VIVES198471}. 
As a result, firms may observe market shocks at different times or with different accuracies, leading to information asymmetry. 
Information asymmetry is further amplified in competitive environments, where some firms invest in more powerful forecasting or other sensing efforts, resulting in better prediction capabilities of market disruption than other firms \citep{shin2010firms}.

When firms rely on private information and respond to market shocks asynchronously, coordination across the supply chain becomes more difficult.
For example, asymmetric information creates challenges for designing effective coordination mechanisms for replenishment and inventory planning policies  \citep{Hafirstpaper,sahin2002flow}. 
This is because staggered information flow causes undersupply during the actual demand phase and oversupply after demand dissipates. 
This was observed during the COVID-19 pandemic, which disrupted both supply and demand, with demand for some products surging while demand for others fell sharply.
Delayed and decentralized responses led firms to join markets or expand production based on incomplete information, resulting in excess inventories and falling prices, as observed in China’s mask market \citep{SixthTone2020}. 
Such situations illustrate how supply chains may have the collective capacity to respond to shocks but fail to do so in a timely or efficient manner because of information asymmetry.

Numerous studies have shown that sharing operational information, such as observed shocks and changing demand or market parameters, 
across supply chain partners allows firms to improve responsiveness, reduce uncertainty, and enhance overall supply chain performance (\cite{cachon2000supply}, 
\cite{lee2000value}, \cite{chen2000quantifying}, \cite{Ivanov_DSCM2021}, \cite{li2017enhancement}, \cite{wong2021conceptual}, \cite{costantino2014impact}, \cite{gu2021impact},\cite{huo2014effects}). 
Beyond overall system improvements, firms with more information may also choose to share it for their own benefit. 
For example, Walmart shares real-time sales and inventory data with suppliers through Retail Link, improving suppliers' demand forecasting and production decisions and strengthening Walmart's competitive advantage.

These well-established benefits of information sharing in supply chains have been studied extensively, yet this literature commonly assumes firms willingly share information or considers information sharing only through bilateral relationships in supply chain contracts \citep{shen2019review}. 
However, in practice, competing firms rarely participate in information exchange platforms.
Consequently, there exists a gap between the proposed benefits of information sharing and the practical reluctance of competing firms to do so.

Given the limited understanding of what drives competing firms to share information \citep{GUAN2020107450}, we study competing firms' decisions to participate in information-sharing platforms. We consider a set of horizontally competing firms that retain full autonomy over whether to join an information-sharing platform and over their operational decisions. 
Firms decide whether to join a centralized information-sharing platform, and this decision is made by firms before knowing whether they or their competitors will acquire private information (what is referred to as stochastic information asymmetry). 
If a firm joins the platform, it commits to sharing information it detects about market conditions, and the platform distributes this information to all participating firms. 
Following the participation decisions, firms make quantity decisions to maximize their own expected profits based on their available information, either their private information, if they do not participate in the platform, or, if they do participate, the shared information through the platform.
The key question for each firm is whether the benefits of information sharing outweigh the loss of competitive advantage from sharing their private information.

We explore equilibrium participation decisions under stochastic information asymmetry and study how firms’ information capabilities, competitive interactions, and network structure affect their incentives to participate in information-sharing. In doing so, this work makes the following key contributions.  
First, we analyze competing firms' information-sharing decisions as strategic decisions made before they know which firms will acquire private information, and thus before they know whether they will act as an information provider or an information recipient.
Second, we characterize equilibrium participation strategies for two horizontally competing firms and show how differences in information power and production efficiency affect participation. When both firms have equal information power, the higher-cost firm is the participation bottleneck. When production costs are equal, the firm with greater information power is the bottleneck.
We also show that market sensitivity affects participation differently under these two asymmetries, motivating participation under cost asymmetry while discouraging it under information power asymmetry. 
Third, we examine information sharing among multiple firms within a networked competition structure, and numerically observe that denser networks reduce participation incentives. In sparse networks, participation is easier among firms with no common markets, while it becomes more difficult when firms compete in overlapping markets.

This paper is organized as follows. Section \ref{Sec:Lit} reviews related literature on information sharing in supply chains.
In Section \ref{Sec:Problem}, we describe the problem setting and present a formulation. Section \ref{Sec:Analysis} presents an analysis of a small supply chain and explores the key factors in firms' participation decisions. 
In Section \ref{Sec:Numerical}, we provide numerical results to explore larger supply chains to further investigate platform participation. Section \ref{Sec:Conc} provides conclusions and directions for future research.

\section{Literature Review} \label{Sec:Lit}

We review related literature in three areas: studies evaluating supply chain information sharing, studies considering information sharing as a strategic decision in bilateral settings, and studies examining information sharing among multiple firms facilitated by a platform.

\subsection{\emph{Value of Supply Chain Information Sharing}} \label{Sec:Lit_exogenous}
In this stream of literature, information sharing is not treated as a decision; instead, this literature evaluates the system-wide benefits of different information sharing policies. 
By exchanging key information such as supplier capacity \citep{choi2008value, yuan2008incentive}, demand estimates  \citep{lee1997information,chen2000quantifying,lee2000value,linhuanglin1046059, ZHAO2002321, CHO201397}, or both \citep{yuan2008incentive}, these studies show that information sharing in supply chains reduces uncertainty, leading to more effective material flow and reduced costs.

When analyzing information sharing, these studies do not consider firms' individual objectives.
Firms are assumed to participate in information sharing regardless of whether they benefit individually. 
When firms maximize individual profits, incentives to share may diverge from system-wide objectives, particularly when private information provides a competitive advantage. 
For example, \cite{zhang2006competition} shows that inventory information exchange can reduce system-wide inefficiencies; yet it does not necessarily benefit all firms. 
Therefore, although related studies show the value of system-wide information sharing, firms' individual objectives and the value of private information must be considered.

\subsection{\emph{Information Sharing as a Decision Problem in Bilateral Settings}} \label{Sec:Lit_decision}

This stream models information sharing as a strategic decision in bilateral settings, where firms exchange information through one-to-one relationships.

A common setting focuses on interactions in which one firm is the sharer and the other firm is the receiver,
generally in vertical interactions where downstream firms hold private demand information that is beneficial to upstream firms. Because the informed firm lacks incentives to share its information voluntarily, a large group of studies examines how contracts can induce information sharing. \citet{shen2019review} provides a comprehensive review. 
For example, manufacturers may purchase demand forecasts from retailers \citep{li2002supply} or offer supply chain contracts that incentivize information sharing \citep{ha2008contracting, ZHOU2017785,Chen04052025}.
Beyond contracts, different mechanisms can also influence information sharing. For example, reducing production costs may motivate retailers to share demand information \citep{hawithcostreduc}. Other studies find that explicit and tacit collusion between retailers may affect sharing incentives \citep{wu2021anti}.

Even when information is shared only on a one-to-one basis, other firms (not part of the one-to-one sharing agreement) introduce additional strategic dynamics into information sharing decisions. For example, when a retailer interacts with multiple upstream firms, the retailer may choose to share with none, one, or both firms, depending on contract timing and operational dynamics, such as nonlinear costs or competition intensity \citep{shang2016information}. Uninformed competitors may infer information indirectly, weakening incentives for voluntary information sharing \citep{zhang2002inference,li2002}. Furthermore, inference effects can even harm the intended information receiver \citep{inference}. Similarly, shared information can be leaked to competitors, although contracts may prevent such leakage \citep{chen2019supply}. In competitive multi-chain environments, retailers’ incentive for sharing private information bilaterally with their respective manufacturers depends explicitly on rivals’ actions. Although mutual information sharing maximizes joint profits, each retailer can prefer to withhold information when the rival shares, leading to a prisoner’s dilemma–type outcome \citep{GUAN2020107450}.

Another group of studies considers settings in which both firms possess private information but the firms decide whether to share information after private information is observed. In such environments, sharing incentives depend on both strategic interactions and signal characteristics. For example, when a supplier and a retailer have signals with different precision, sharing incentives depend on the variance and correlation of their signals \citep{zhang2013coordination}. In a coopetitive supply chain where one of the firms serves both as the supplier and the competitor of the other firm, sharing information may benefit one firm while harming the other \citep{wang2020}. 
\cite{RAWEEWAN2018269} develops a static game-theoretic model of bilateral information sharing among competing firms, where firms simultaneously decide whether to reveal private signals in a static bilateral information sharing game with exogenously determined payoffs.

While this literature extensively examines strategic bilateral information sharing under asymmetric information and competition, all known work assume two key features. First, firms decide whether to share information after private information is observed, allowing them to make sharing decisions conditioned on realized informational advantage. Second, the interaction structure is predefined: the set of firms and the direction of information flows are given. 
Hence, these studies consider the environments in which the firms already have their private information when making information sharing decisions. 
In many supply chain environments, however, informational advantage varies across realizations, and firms need to make participation decisions in information sharing before observing the information. 
In these cases, participation decisions differ significantly from information sharing decisions based on realized signals. 

\subsection{\emph{Platform-Based Information Sharing Among Multiple Firms}} \label{Sec:Lit_platform}

Because information sharing through bilateral relationships is often difficult to manage when multiple firms operate in the supply chain, a stream of studies examines multi-agent settings, in which a platform facilitates information exchange among multiple firms simultaneously.

Several studies analyze settings where a platform possesses superior information and selectively reveals it to competing firms \citep{Hou28012026, Li2021online, ZHONG2023, zha2023strategic, liu2021information}. 
Online retail platforms are a common example in this setting. In these models, the platform determines the level of information to reveal to each participant, while firm participation is assumed to be given. 
The focus is on how such a platform affects competition and operational outcomes, rather than on firms’ incentives to join the platform. 
\cite{MASSARI2021108150} studies information sharing among competing firms using agent-based simulation, and finds that information sharing can improve performance even under competitive environments. 
However, participation is assumed rather than derived from equilibrium behavior, so firm-level incentives are not explicitly modeled.

These studies highlight the role of platforms in facilitating information sharing among multiple competing firms. However, they assume firms will participate and do not consider their private incentives. 
Two notable exceptions are \cite{Wu01022022platform} and \cite{liu2026mechanism}, both of which study participation decisions in an information-sharing platform. In both settings, each participant has private demand information. Therefore, each competitor knows that, conditioned on participation, it will both share its private information and receive information from other participants.  Specifically, \cite{liu2026mechanism} study privacy-preserving information sharing among firms competing in a single Cournot market. Their platform aggregates participants' private demand signals, considering a privacy preserving mechanism, with its own external information to provide back to participants an aggregate demand estimate. 
\cite{Wu01022022platform} studies participation decisions in an information-sharing platform among competing garages with interdependent demands. In their setting, each garage has complete information about its own demand shock, and platform participation allows it to access competitors' private demand information. 
In contrast, our work is interested in settings where firms do not yet know whether they will share or receive information when they decide participation. 
Although voluntary platform participation has been considered in these two papers, the conditions that shape participation and platform formation under stochastic information asymmetry, where the firms are uncertain about who will have private information, has not been studied.

\section{Problem Definition and Formulation} 
\label{Sec:Problem}

We study firms' participation decisions in an information-sharing platform within a quantity-competition setting in a perishable-product supply chain prone to shocks. 
Firms decide whether to participate before observing the market state and without knowledge of other firms’ private information. Therefore, platform participation is considered a strategic decision.
A firm joins the platform only if participation yields a higher expected profit. 
Since one firm's participation affects the information available to others and their quantity decisions, we model the participation decision in the information-sharing platform as a two-stage dynamic game.
In the first stage, firms decide whether to participate, and in the second stage, they make their quantity decisions given the realized platform configuration.
Our goal is to analyze the information sharing participation decisions under different supply chain settings. 
In the following part, we first present the underlying supply chain environment. We then introduce the two-stage game structure and the information-sharing platform, and formulate the second-stage quantity decision and the first-stage participation decision.

\subsection{\emph{Underlying Supply Chain Environment}}\label{Sec:Supply_chain_setting}

Consider a set of firms $\mathcal{F}$ in a quantity competition across a set of markets $\mathcal{M}$. 
Each firm serves a subset of markets. 
We model firms' competition in markets 
with a two-echelon supply chain as a bipartite network $G=(\mathcal{F} \cup \mathcal{M}, E)$, where $E$ is the set of firm-market edges. 
An edge between a firm $i\in \mathcal{F}$ and market $j\in \mathcal{M}$ indicates that Firm $i$ serves Market $j$. 
In a given supply network $G$, Firm $i$ 
serves a subset of markets $M_i:=\{j \in \mathcal{M}\mid(i,j) \in E\}$, and Market $j$ is served by a subset of firms $F_j:=\{i \in \mathcal{F}\mid(i,j) \in E\}$.

Firms supply a homogeneous seasonal or perishable product, e.g., fresh food or consumer goods with a short shelf life, to markets. Due to the perishable nature of the product, firms do not carry over inventory. Therefore, production decisions are single-period decisions \citep{cachon2011value}. 
In each production period, Firm $i$ chooses the quantity $q_{ij}\geq 0$ supplied to each market $j\in M_i$. Markets operate under clearing conditions; all supplied products are sold within the operational cycle, and firms' quantity decisions directly determine market prices; in particular, higher total supply in a market leads to a lower price \citep{wang2009markdown}. Hence, following the platform participation decision, firms are in a Cournot competition.

Let $Q^j:=\sum \limits_{i\in F_j}q_{ij}$ denote the total supply in Market $j$, 
for each market $j\in \mathcal{M}$, the market-clearing price is
\begin{equation}
P_j(Q^j)=a_j -b_j\cdot Q^j,
\tag{Baseline market price}
\end{equation}
where $a_j>0$ is the price intercept of Market $j$, and $b_j>0$ measures how sensitive the price is to total supply in the market. 
We refer to these parameters as \textit{market size} and \textit{market sensitivity}, respectively.
The linear price function follows standard market-clearing models widely used in the Cournot competition literature \citep{CAO2013425, fang2015managing, bimpikis2019cournot}.

We assume convex production costs, reflecting a soft production capacity for the firms \citep{bimpikis2019cournot}. Specifically, Firm $i$'s total production cost is
\begin{equation}
        C_i\left(\sum \limits_{j\in M_i}q_{ij}\right)=c_i\cdot \left(\sum \limits_{j\in M_i}q_{ij}\right)^2, 
    \end{equation}
where $c_i>0$ is a firm-specific cost parameter that captures Firm $i$'s \textit{production efficiency} or \textit{cost competitiveness}.

\noindent\textbf{Impact of Shocks:} The supply chain environment described above represents a \textit{baseline case} characterized by the market-clearing price parameters $(a_j,b_j)$. 
However, supply chain disruptions can affect these parameters.
The supply chain may experience a shock with probability $\omega$. 
Conditional on a shock occurring, we assume that it affects a single Market $j$ with probability $r_j>0$, where $\sum\limits_{j\in \mathcal{M}}r_j=1$, providing a tractable modeling approach \citep{gao2019disruption,valuedisruptionduration,Tomlin,SEREL2008115}. 
Indicator $\theta \in \mathcal{M} \cup\{0\}$ denotes the realized market state, where $\theta=0$ represents no shock and $\theta=j$, $j\in \mathcal{M}$ represents a shock occurring at Market $j$. 
Accordingly, the prior distribution over the market state can be expressed as 
\begin{equation*}
       \Pr(\theta=0) = 1-\omega  \quad (\text{no shock}) \qquad \text{and} \qquad \Pr(\theta=j) =\omega r_j  \quad (\text{Shock in Market $j \in \mathcal{M}$}).
\end{equation*}

 A shock in a market affects the price dynamics, shifting its price intercept by $D$. 
Therefore, the state-dependent prices under the market-shock effect are represented as follows
\begin{equation}
    P_{j}(Q^j|\theta)= a_j+D\cdot\,\mathbf{1}\{\theta=j\}-b_j\cdot Q^j,
    \tag{State-dependent market price }
    \end{equation}
where $D>0$ is a constant parameter, representing the \textit{shock magnitude}.

When a shock occurs, firms may receive private information before making production decisions. Specifically, Firm $i$ receives an independent \textit{shock signal} with probability $p_{i}$, reflecting its \textit{information power}. 
The information power of each firm is common knowledge, but firms do not observe whether other firms have received a signal or not.
A firm can observe a signal only when there is actually a shock. 
That is, when there is no shock, firms do not receive a false signal that might lead them to believe there is one.
We assume that the observed signal is accurate. 

Firm $i$’s profit equals its total revenue from the markets it serves minus the production cost. For a given market state $\theta$, Firm $i$’s profit function is expressed as
\begin{equation}
\pi_i(\vec{q},\theta)=\sum_{j \in M_i} P_j(Q^j \mid \theta)q_{ij}-C_i\left(\sum \limits_{j\in M_i}q_{ij}\right).
\label{Eq:Profit}
\end{equation}
where $\vec{q}=(q_{kl})_{(k,l)\in E}$ is the vector of quantities supplied by firms to markets.

Private signals create information asymmetry, leading to a Bayesian Cournot competition in quantity decisions.

\subsection{\emph{Information-sharing Platform}}

Market shocks create challenges for efficient quantity decisions, leading to under- or oversupply. 
The information-sharing platform aims to improve supply chain performance by reducing information asymmetry among participating firms through the sharing of private signals. 
Firms’ participation decisions can be driven by two motives: improving their own access to market information to support better decision-making and reducing uncertainty regarding other firms' information state. If every firm observes the market shock but is unaware of other firms' observations, they will produce more than what they would under complete information. This creates oversupply in the market, which can be avoided by information sharing.

The process that governs the firms' platform participation decision can be modeled as a two-stage dynamic game.
In the first stage, firms simultaneously decide whether to participate in the information-sharing platform before observing the realized market state or other firms’ private information. 
Each set of participating firms represents a unique \textit{platform participation configuration} (hereafter referred to as \textit{configuration}), denoted by a binary vector $\vec{V} \in \{0,1\}^{|\mathcal{F}|}$, where $V_i=1$ indicates that Firm $i$ participates in the platform.

Once the platform is formed, in the second stage, a market shock may occur, and firms choose their quantities based on the observed information conditional on the realized platform configuration $\vec{V}$ in a networked Cournot competition under uncertainty.
Nonparticipating firms rely on their own private information, while participants share the platform's accumulated information power, denoted by $p_{\text{platform}}(\vec{V})$ for a platform configuration $\vec{V}$, defined as the probability that at least one participant detects an existing market shock. 
According to the independent detection assumption introduced earlier, platform $\vec{V}$'s detection probability is
\begin{equation}
    p_{\text{platform}}(\vec{V})
= 1 - \prod_{i \in \mathcal{F}} (1 - p_i)^{V_i}. \tag{Platform's information power}
\end{equation}
Therefore, for any participating Firm $i$ (i.e., $V_i = 1$), we have $p_{\text{platform}}\geq p_i$.

The sequence of events and decisions is presented in Figure~\ref{fig:timeline}.

\begin{figure}[h]
\centering
\begin{tikzpicture}[>=stealth, font=\normalsize]

    \draw[thick,->] (0,1.25) -- (16,1.25);

    \draw[thick] (2,1.25) -- (2,1);
    \draw[thick] (6,1.25) -- (6,1);
    \draw[thick] (10,1.25) -- (10,1);
    \draw[thick] (14,1.25) -- (14,1);

    \node[align=center] at (2,2.1) {\textbf{First-Stage} \\ \textbf{Decisions}};
    \node[align=center] at (2,0) {Firms decide \\whether to join\\ the platform.};

    \node[align=center] at (6,0) {Shock \\realization};

    \node[align=center] at (10,0) {Private signal \\observation. Participants\\ share signals.};

    \node[align=center] at (14,2.1) {\textbf{Second-Stage} \\ \textbf{Decisions}};
    \node[align=center] at (14,0) {Firms choose\\ their production\\ quantities.};

\end{tikzpicture}
\caption{Sequence of events for platform participation decisions in the first stage and the quantity decisions in the second stage.}
\label{fig:timeline}
\end{figure}
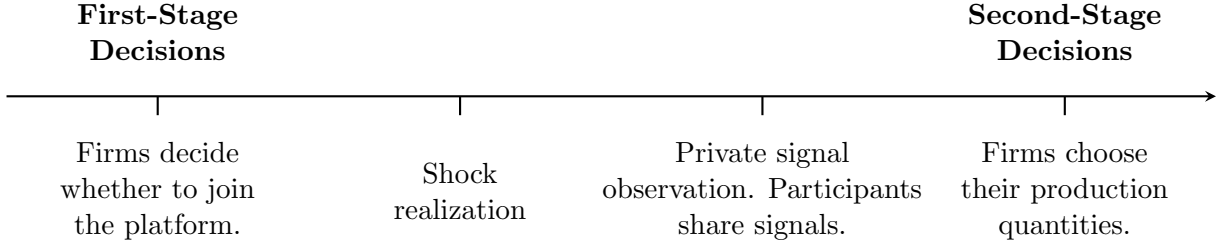

To illustrate the information-sharing platform in more detail, we present an example.

\textbf{Example. }
\emph{Consider a small supply chain with three firms and two markets as shown in Figure \ref{Fig:network}. Suppose that, in the first stage, only Firms 2 and 3 join the platform, represented by platform configuration $\vec{V}=[0,1,1]$. 
Once the platform is formed, the supply chain will operate as usual, experiencing occasional market shocks.
While the information power of the platform formed by Firms 2 and 3 is constant, the actual information may differ depending on the realized shock and its observability. The ``Possible Realizations'' row in Figure \ref{Fig:platform} illustrates all nine possible cases based on the realizations of shock occurrence, signal detection by each firm, and information sharing under this platform configuration.}

\begin{figure}[h]
    \centering
    \begin{subfigure}[t]{0.20\textwidth}
        \centering
        \includegraphics[width=0.8\textwidth]{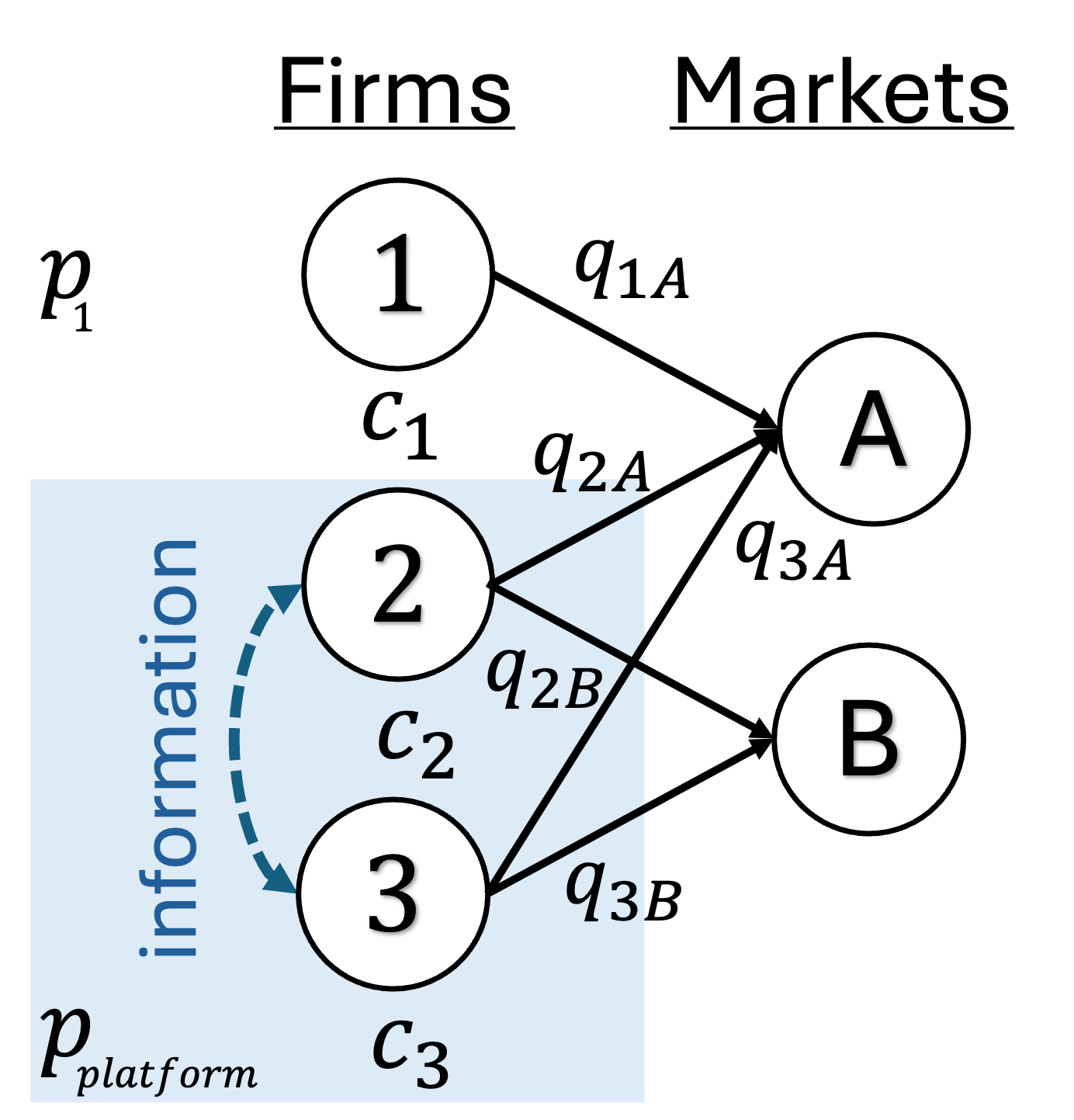}
        \caption{Small supply chain network.}
        \label{Fig:network}
    \end{subfigure}
    \hfill
    \begin{subfigure}[t]{0.7\textwidth}
        \centering
        \includegraphics[width=\textwidth]{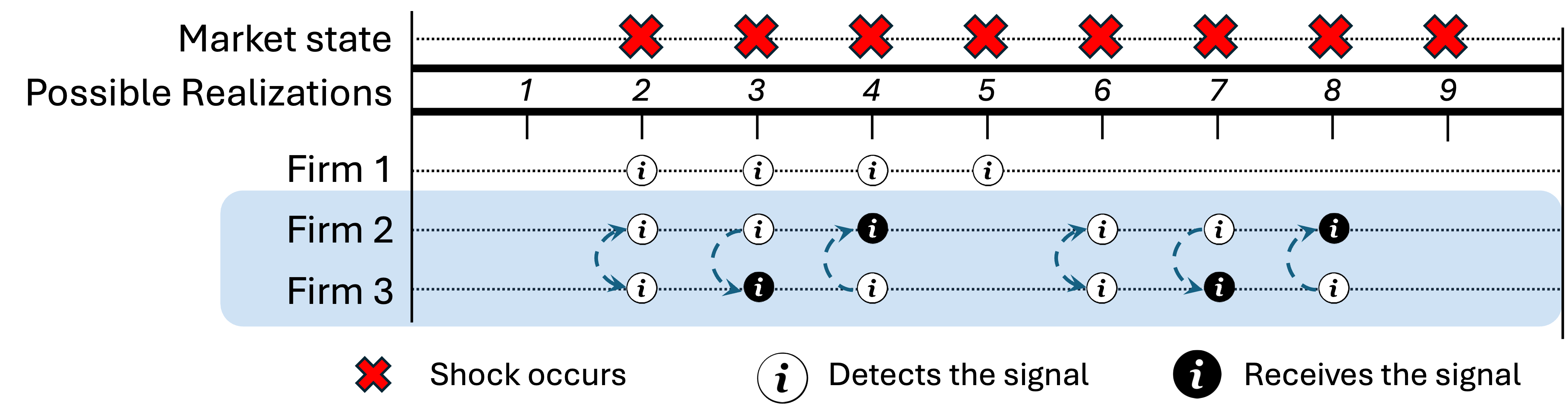}
        \caption{Possible realizations of firms' signals and information exchange among participants in an illustrative example.}
        \label{Fig:platform}
    \end{subfigure}
    \hfill

    \caption{Example for information sharing with a platform: Structural (left) and operational (right) views of a three-firm two-market example supply chain, where Firms 2 and 3 participate in the platform (i.e., $\vec{V}=[0, 1, 1]$). }
    \label{Fig:sample}
\end{figure}

\emph{We explain the trade-offs of platform participation from the perspective of Firm 3. In realizations 1, 5, and 9, neither Firm 2 nor Firm 3 observes a shock signal. Therefore, Firm 3's belief that `the market has no shock' is stronger than its belief about the same condition if it were not on the platform. 
In realizations 3 and 7, Firm 3 receives the shock signal through information sharing. In realization 7, participation gives Firm 3 an advantage over the non-participating Firm 1, which remains unaware of the shock.   
In all realizations, Firm 3 knows Firm 2's information state (and vice versa), which reduces uncertainty in Firm 3's decision problem. Particularly in realizations 2 and 6, Firm 3 knows Firm 2 also has the signal and adjusts its quantity decisions to avoid overproduction.}

In the following part, we formulate firms' quantity decisions in the second stage under a given platform configuration (including the no-platform case) as a Bayesian Cournot competition. Using this model, we then determine equilibrium platform configurations.

\subsection{\emph{Firms' Quantity Decision}}\label{Sec:bayesian_game_formulation}%
Because products are perishable and there is no inventory carryover or backordering, firms' quantity decisions are independent across time periods.  
Therefore, we model the second-stage quantity competition as a static Bayesian quantity game that captures uncertainty from market shocks and asymmetric information. 
The game structure depends on both the supply chain network and the platform participation configuration. 
We denote the corresponding Bayesian game under a specific information-sharing platform configuration $\vec{V}$ 
by $\Gamma(\vec{V})$. 
Thus, our model can be viewed as a family of Bayesian games.
The game elements are as follows.

\textit{\underline{Firm Types:}} In the Bayesian quantity game, a firm's type includes the market shock information it has prior to making its quantity decision.
Firms participating on the platform, if it forms, are of the same type because they share the same information. Hence, a firm’s type depends on the realized shock information and the platform configuration.

Let $\mathcal{T}$ be the set of all firm types; we denote firm $ i$'s type by $z_i$.
Type $z_i = ( +, j )$ indicates that Firm $i$ has a shock signal for Market $j$, either through its own observation or through information sharing. Type $z_i = (-)$ implies that Firm $i$ does not have any signal. Let $\vec{z}_{-i}$ denote the type profile of all firms except Firm $i$, and let $\vec z=(z_i,\vec z_{-i})$ denote the type profile of all firms.

\textit{\underline{Beliefs:}} Each firm forms beliefs over the market state $\theta$ and the types of other firms $\vec{z}_{-i}$. 
These beliefs are based on common prior knowledge about each firm and the market, as well as on the firm's private information and information sharing induced by the platform configuration $\vec{V}$. 
Note that firms on the platform share the same type as other participants. 
Firm $i$'s posterior beliefs over the system state are represented by the conditional joint distribution $Pr(\theta,\vec{z}_{-i}\mid z_i,\vec{V})$.  
Explicit expressions for the prior and posterior probabilities are provided in Appendix \ref{Sec:Appx:probabilities}.

\textit{\underline{Actions:}} 
Based on its belief, each firm decides the quantities it supplies to the markets it serves.
We denote the production quantities of Firm $i$ by $q_i(z_i;\vec V)=(q_{ij}(z_i;\vec V))_{j \in M_i}$ and the full quantity vector by
$\vec q (\vec z;\vec V)= (q_{ij}(z_i;\vec V))_{(i,j)\in  E}$.

\textit{\underline{Payoff Function:}} In the Bayesian quantity game $\Gamma(\vec{V})$, a firm's payoff is its expected realized profit over the firm's belief about the market state and firm types. By taking the expectation over $\theta$ and $\vec z_{-i}$ in the profit in Equation \ref{Eq:Profit}, the payoff of Firm $i$ given type $z_i$ is

\begin{equation}
\begin{aligned}
\sum_{\theta\in \{0\}\cup \mathcal{M}}
\sum_{ \vec{z}_{-i}}
Pr(\theta,\vec{z}_{-i}\mid z_i,\vec{V})\,
 \cdot \pi_i\!\left(\vec q \;(\vec z;\vec V),\theta\right). 
\end{aligned}
\end{equation}

\textit{\underline{Bayesian Nash Equilibrium:}}
A collection of strategies $q^*_i(z_i;\vec V)$, $i\in \mathcal F$ and $z_i\in\mathcal T$, is a Bayesian Nash equilibrium if, for every firm $i$ and its type $z_i$, the following condition holds:
\begin{equation}
q_{i}^*(z_i;\vec{V})\in
\operatorname*{arg\,max}_{q_{i}} \mathbb{E}\!\left[\pi_i\!\left(\vec q \;(\vec z;\vec V),\theta\right)\mid z_i,\vec{V}\right],
\end{equation}
where, in the maximization, $\vec{q}$ consists of Firm $i$'s quantity choices $q_i$
and the type-dependent equilibrium quantities $q_k^*(z_k;\vec{V})$ of all other firms $k \neq i$.

By solving all firms' maximization problems simultaneously, we find the equilibrium quantities.
For a given type profile $\vec{z}$, we denote the equilibrium quantity vector under platform configuration $\vec{V}$ by
\begin{equation}
\vec{q}^{\,*}(\vec{z};\vec{V})
=
\left(q^*_{ij}(z_i;\vec{V})\right)_{(i,j)\in E}.
\end{equation}

\subsection{\emph{Platform Participation Decision}}

When making the platform participation decisions in the first stage, a firm's payoff is the expectation of its profit over all realizations of shocks and signals when firms play the Bayesian Nash equilibrium of the corresponding quantity decisions $\Gamma (\vec{V})$, given by
\begin{equation}
\Pi_i(\vec{V})
= \sum_{\theta\in \{0\}\cup \mathcal{M}}
\sum_{\vec z\in \mathcal{T}^{|\mathcal F|}}
\Pr(\theta,\vec z\mid \vec{V})\,
\cdot
\pi_i\!\left(\vec{q}^{\,*}(\vec z;\vec{V}),\theta\right).
\end{equation} 

A platform configuration $\vec{V}^*$ is an equilibrium if none of the firms can increase its payoff by unilaterally changing its participation decision.%

To analyze the participation decisions, we solve the family of Bayesian quantity games $\{\Gamma(\vec{V})\}_{\vec{V}\in \{0,1\}^{|\mathcal{F}|}}$ and compute the associated payoffs $\Pi_i(\vec{V})$ for all firms to identify no-incentive-to-deviate cases.
Since the number of possible configurations is exponential in the number of firms, identifying the set of equilibria is computationally demanding.
Therefore, in the next section, we first analyze a small supply chain to derive insights.
Then, we explore how these insights extend to larger supply chains, where network structure creates additional strategic interactions, through a numerical study in Section \ref{Sec:Numerical}.

\section{Equilibrium Analysis on Small Supply Chains}
\label{Sec:Analysis}
We consider a small supply chain setting where two firms serve a single market.
This setting provides a baseline that isolates the information sharing tradeoff from network effects: participation improves a firm's information availability but also reveals information to the competitor. We characterize how shock attributes, firms' information and competition power affect firms' participation incentives for homogeneous firms, and for firms that differ in production costs or information power. 
We present the production quantities used in the following results in Appendix \ref{App:Equilibrium}.

In the two-firm case, unilateral participation does not form a platform; hence, only two platform configurations are relevant for equilibrium: (i) the baseline supply chain, $\vec{V}=[0, 0]$, and (ii) both firms join the platform and share their demand signals, $\vec{V}=[1, 1]$. 
Configurations $[0, 1]$ and $[1, 0]$ are the same as $[0, 0]$, and $[0,0]$ is always an equilibrium.

The \emph{participation gain} for Firm $i$ is equal to the expected net benefit of information sharing and is computed as follows:
\begin{equation}\Delta_i:=\Pi_i([1,1])-\Pi_i([0,0]).
\label{Eq:Gain}\end{equation}
The platform forms when $\Delta_{i} \geq 0$ for $i\in \{1, 2\}$.

We present the derivation of $\Delta_i$ and the firms' participation decision analysis in Appendix \ref{App:participation_gain_derivation}.
In the following part, we analyze the participation decisions in terms of firms' characteristics and market parameters. 
All proofs are in Appendix \ref{App:participation_proofs}.

\subsection{\emph{Homogeneous Firms}}

In the homogeneous case, firms are identical in both information power and production costs, i.e.,  ($p_1 = p_2=p$ and $c_1 = c_2=c$). 
In this case, since firms' participation decisions are symmetric, the platform forms when either firm has an incentive to join. 

Proposition \ref{prop:homogeneousregions} characterizes the conditions under which the information-sharing platform is formed.

\begin{proposition}[Characterization of platform participation]
\label{prop:homogeneousregions}
Homogeneous firms participate in an information-sharing platform if and only if at least one of the following conditions holds:
\begin{enumerate}[label=(\roman*)]

\item \textbf{High cost and low market sensitivity:} $c \geq \frac{1}{2(1+\sqrt{2})}\,b_A$.

\item \textbf{High information power:} $ p \geq
\frac{1}{2}\sqrt{\frac{8+13y}{y}}
-\frac{4+3y}{2y}$,
where
$y=\frac{b_A}{c}$.

\item \textbf{High shock probability:}
$
\omega \geq \bar \omega(y,p)= 
\frac{
4+4y(1+p)-y^2\!\left(1-3p-p^2\right)
}{
(3y+2)p\left(2+y(-1+2p)\right)
}$, 
where 
$y=\frac{b_A}{c}$.

\end{enumerate}

\end{proposition}

\begin{figure}[htbp]
    \centering

    \begin{subfigure}[t]{0.31\textwidth}
        \centering
        \includegraphics[width=\textwidth]{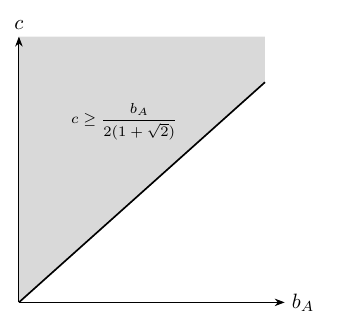}
        \caption{High cost and low market sensitivity}
        \label{Fig:bc_region}
    \end{subfigure}
    \hfill
    \begin{subfigure}[t]{0.31\textwidth}
        \centering
        \includegraphics[width=\textwidth]{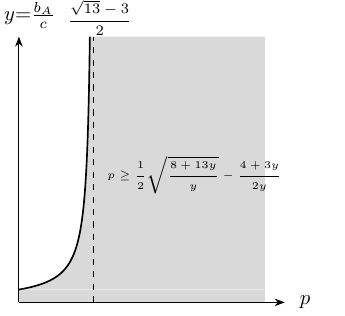}
        \caption{High information power}
        \label{Fig:py_region}
    \end{subfigure}
    \hfill
    \begin{subfigure}[t]{0.31\textwidth}
        \centering
        \includegraphics[width=\textwidth]{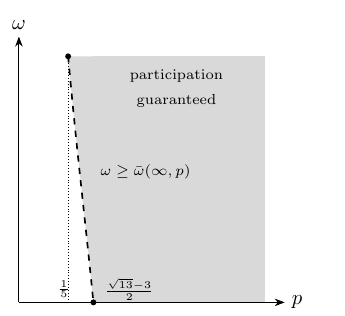}
        \caption{High shock probability}
        \label{Fig:pomega_region}
    \end{subfigure}

    \caption{Illustration of Proposition \ref{prop:homogeneousregions}: Participation conditions for homogeneous firms to form a platform. Platform forms in shaded regions.} 
    \label{Fig:homogeneous_participation_regions}
\end{figure}

These conditions identify when the benefit of improved information outweighs the strategic cost of revealing private information. 
We illustrate the regions defined by these conditions in  Figure~\ref{Fig:homogeneous_participation_regions}.

(i) \textbf{High cost and low market sensitivity (Figure \ref{Fig:bc_region}):} Participation is guaranteed when production costs are sufficiently high relative to market sensitivity (higher $c$ or lower $b$). 
With high production costs, the competitor's reaction to shared information is limited. With low market sensitivity, even when the competitor increases its production in response, the price does not change significantly.
These effects diminish the intensity of competition, reducing the strategic cost of sharing information and making platform participation more attractive.

(ii) \textbf{High information power (Figure \ref{Fig:py_region}):}
  When both firms have high information power, sharing information does not lead to a significant informational advantage, since the competitor is likely to observe the shock as well. Instead, the platform helps firms resolve uncertainty about each other's types, enabling them to choose production quantities based on more accurate beliefs. 

(iii) \textbf{High shock probability (Figure \ref{Fig:pomega_region}):} When the shock is a low-probability event, the signals become strategically valuable, and firms prefer to keep it private to maintain informational advantage. On the other hand, as the probability of a shock increases, the value of private information decreases, while the benefit of eliminating information asymmetry remains, making firms more willing to share information through the platform.

In summary, under conditions (ii) and (iii), the private signal has limited value, whereas under condition (i), the adverse economic consequences of information disclosure are limited for homogeneous firms.

\subsection{\emph{Heterogeneous Firms}} \label{Sec:heterogeneous}

The analysis of homogeneous firms shows how supply chain parameters interact to shape participation decisions. However, firms often differ in production efficiency or information power. We therefore examine how these asymmetries affect participation incentives. Because the platform requires both firms to participate, we focus on the firm with the weaker incentive, i.e., the bottleneck for platform formation. We examine the two sources of asymmetry separately, varying one parameter at a time to isolate their effects.

\subsubsection*{Cost Asymmetric Firms ($p_1 = p_2=p$ and $c_1 \neq c_2$):} 
We compare the participation incentives of cost-efficient and -inefficient firms to identify the bottleneck firm and analyze its participation conditions.

We first show that when firms differ only in cost efficiency, neither market size nor disruption magnitude affects platform formation, as in the homogeneous-firm case.

\begin{proposition}[Participation under same information power]
\label{cor:positive_D}
In two-firm, single-market supply chains where firms have the same information power, neither the market size nor the (positive) shock magnitude changes platform participation decisions. 
\end{proposition}

Following this observation, we examine participation incentives as a function of other supply chain parameters.
In the rest of this section, without loss of generality, we will assume that Firm $1$ is the cost-efficient firm, i.e., $c_1 < c_2$. 
When firms differ only in cost, the more cost-efficient firm benefits more from shared information because it can increase its production quantity at a lower cost and is therefore less affected by the competition.
The following proposition shows that Firm 2 is the bottleneck for platform formation. 

\begin{proposition}[Participation bottleneck for cost asymmetric firms]\label{prop:cost_bottleneck}
Suppose that Firm $1$ is more cost-efficient than Firm $2$, with $c_1<c_2$.
Then, $\Delta_2\geq 0$ implies that  $\Delta_1>0$.
That is, whenever the cost-inefficient firm weakly benefits from
participating in the information-sharing platform, the cost-efficient
firm strictly benefits from participation.
\end{proposition}

Because the cost-inefficient firm is the bottleneck, conditions that support its participation also support platform formation. We therefore characterize sufficient conditions for each firm's participation in the $(p,\omega)$ space and compare the corresponding regions for the cost-efficient and cost-inefficient firms in Figures \ref{fig:asymmetric_cost_firm1} and \ref{fig:asymmetric_cost_firm2}, respectively. The derivations are provided in Appendix \ref{App:inefficientplot}.
\begin{figure}[h]
    \centering

    \begin{subfigure}[t]{0.4\textwidth}
        \centering
        \includegraphics[
            height=6cm,
            keepaspectratio
        ]{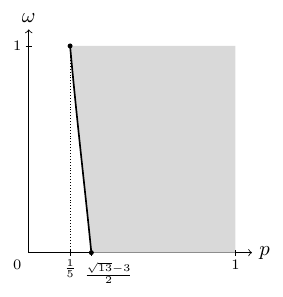}
        \caption{Cost-efficient Firm 1.}
        \label{fig:asymmetric_cost_firm1}
    \end{subfigure}
    \hfill
    \begin{subfigure}[t]{0.55\textwidth}
        \centering
        \includegraphics[
            height=6cm,
            keepaspectratio
        ]{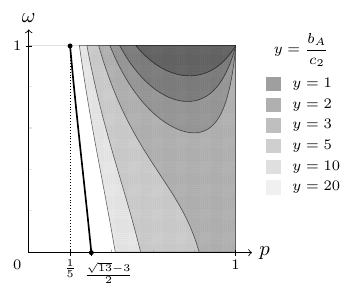}
        \caption{Cost-inefficient Firm 2 for different values of
        $\frac{b_A}{c_2}$.}
        \label{fig:asymmetric_cost_firm2}
    \end{subfigure}

    \caption{Sufficient regions for participation in the $(p,\omega)$ parameter space.
    The shaded regions indicate parameter combinations for which participation is guaranteed.}
    \label{fig:asymmetric_cost}
\end{figure}
We observe that the participation of the cost-inefficient firm changes monotonically with market sensitivity (Figure \ref{fig:asymmetric_cost_firm2}). As $\tfrac{b_A}{c_2}$ increases, the participation region expands, reducing the role of the cost asymmetry in the platform formation.
This is because, as market sensitivity increases relative to the cost-inefficient firm's production cost, producing more becomes less desirable for both firms primarily due to its impact on the market price (i.e., reducing the price at a higher rate).
This leaves less room for the shared information to actively influence the firms' quantities. 
Consequently, the strategic cost of sharing private information decreases. 
Moreover, information sharing removes uncertainty about the other firm's type, enabling firms to predict each other's production decisions more accurately. 
These effects collectively make information sharing more preferable for the cost-inefficient firm.

The effect of market sensitivity thus differs across the two settings:
higher market sensitivity encourages platform participation when firms
differ in cost, whereas lower market sensitivity encourages participation when firms are homogeneous.

\subsubsection*{Information Power Asymmetric Firms ($p_1 \neq p_2$ and $c_1 = c_2=c$):} We now analyze platform participation decisions under information power heterogeneity only. 
Without loss of generality, we assume Firm $1$ has higher information power, i.e., $p_1>p_2$. 
The following proposition shows that Firm 1 is the bottleneck and determines the platform formation.

\begin{proposition}[Participation bottleneck for information power asymmetric firms]\label{prop:infopower_bottleneck}
When firms differ only in information power and $p_1>p_2$, then $\Delta_1\geq 0$ implies that $\Delta_2>0$.
That is, whenever the high-information power firm weakly benefits from participating in the information-sharing platform, the low information power firm strictly benefits from participation.
\end{proposition}

Information power asymmetry makes the analysis much less tractable. 
We therefore adopt a numerical approach to identify when Firm $1$ participates and, by Proposition \ref{prop:infopower_bottleneck}, when the platform forms.
We explore the entire $(p_1,\tfrac{p_2}{p_1})$ parameter space for different values of $\tfrac{b_A}{c}$ and $\omega$. We focus on these parameters because participation decisions are particularly sensitive to them and vary little with the remaining parameters.
We illustrate selected results in Figure \ref{Fig:asymmetric_information} for $\omega=0.2$.  
In each plot, the shaded region represents where the platform forms.

\begin{figure}[h]
    \centering
    \begin{subfigure}{0.3\textwidth}
        \centering
        \includegraphics[width=\textwidth]{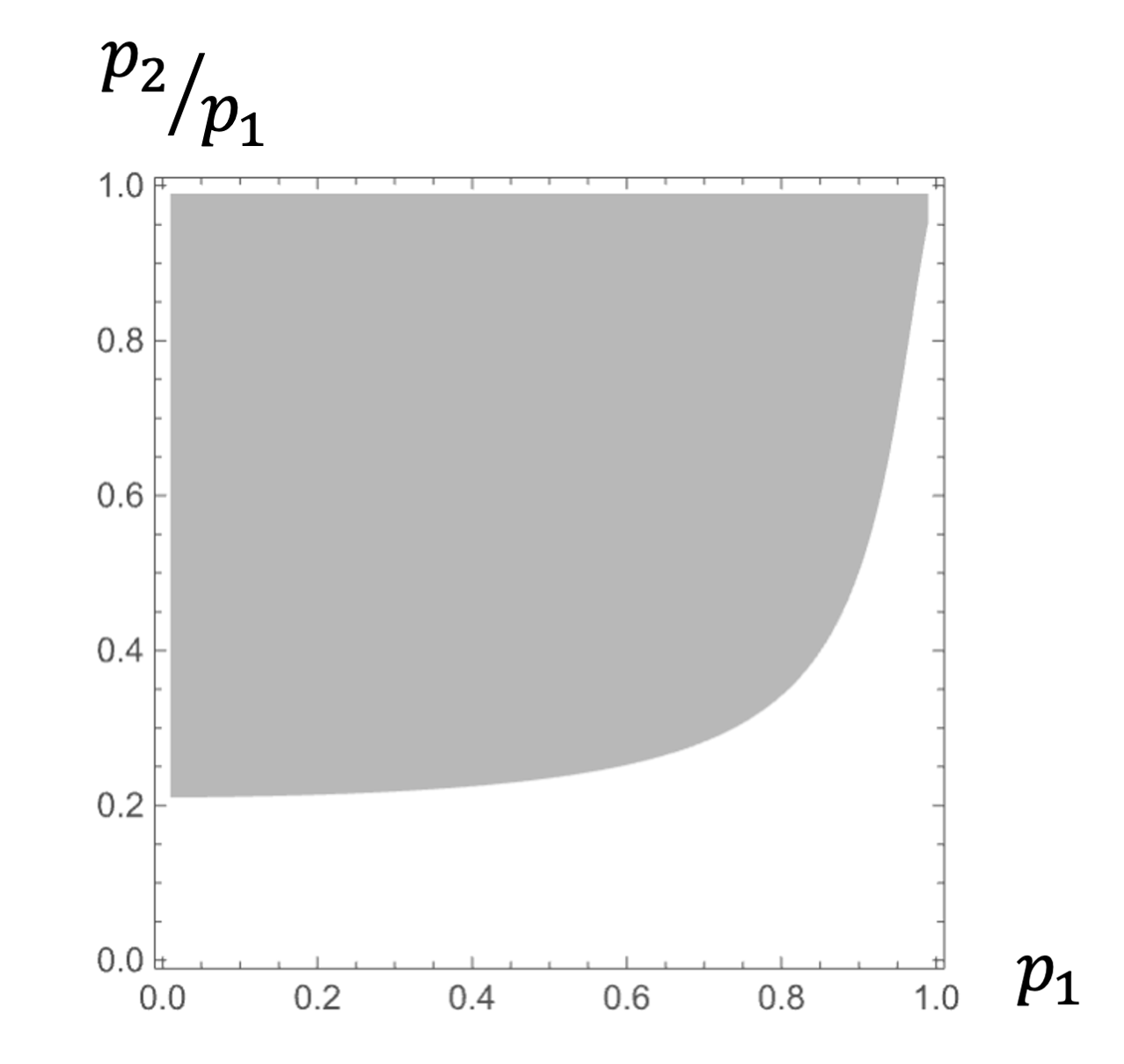}
        \caption{$b_A/c=0.25$}
        \label{Fig:asymmetric_information_025}
    \end{subfigure}
    \hfill
    \begin{subfigure}{0.3\textwidth}
        \centering
        \includegraphics[width=\textwidth]{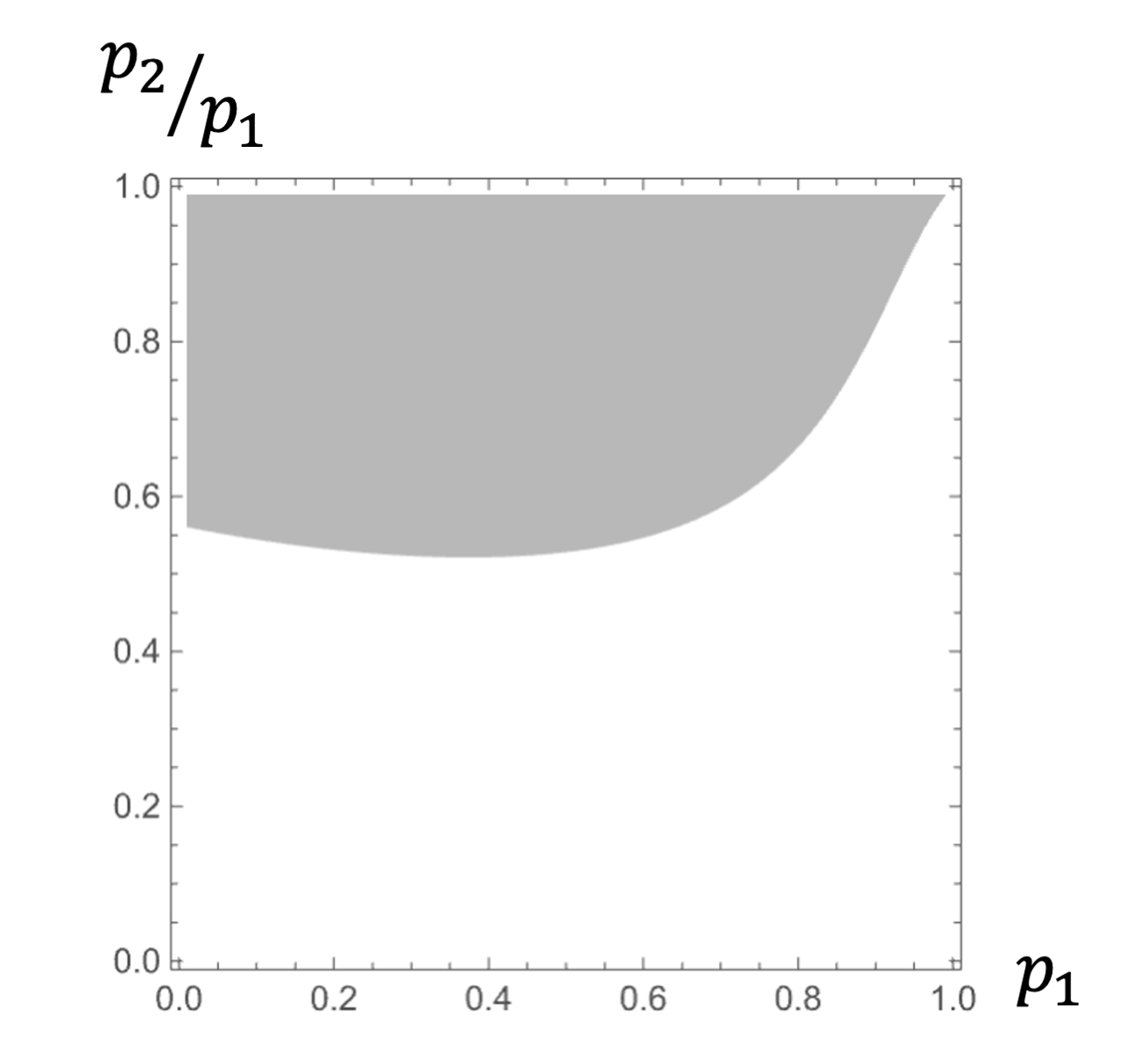}
        \caption{$b_A/c=1$}
        \label{Fig:asymmetric_information_1}
    \end{subfigure}
    \hfill
    \begin{subfigure}{0.3\textwidth}
        \centering
        \includegraphics[width=\textwidth]{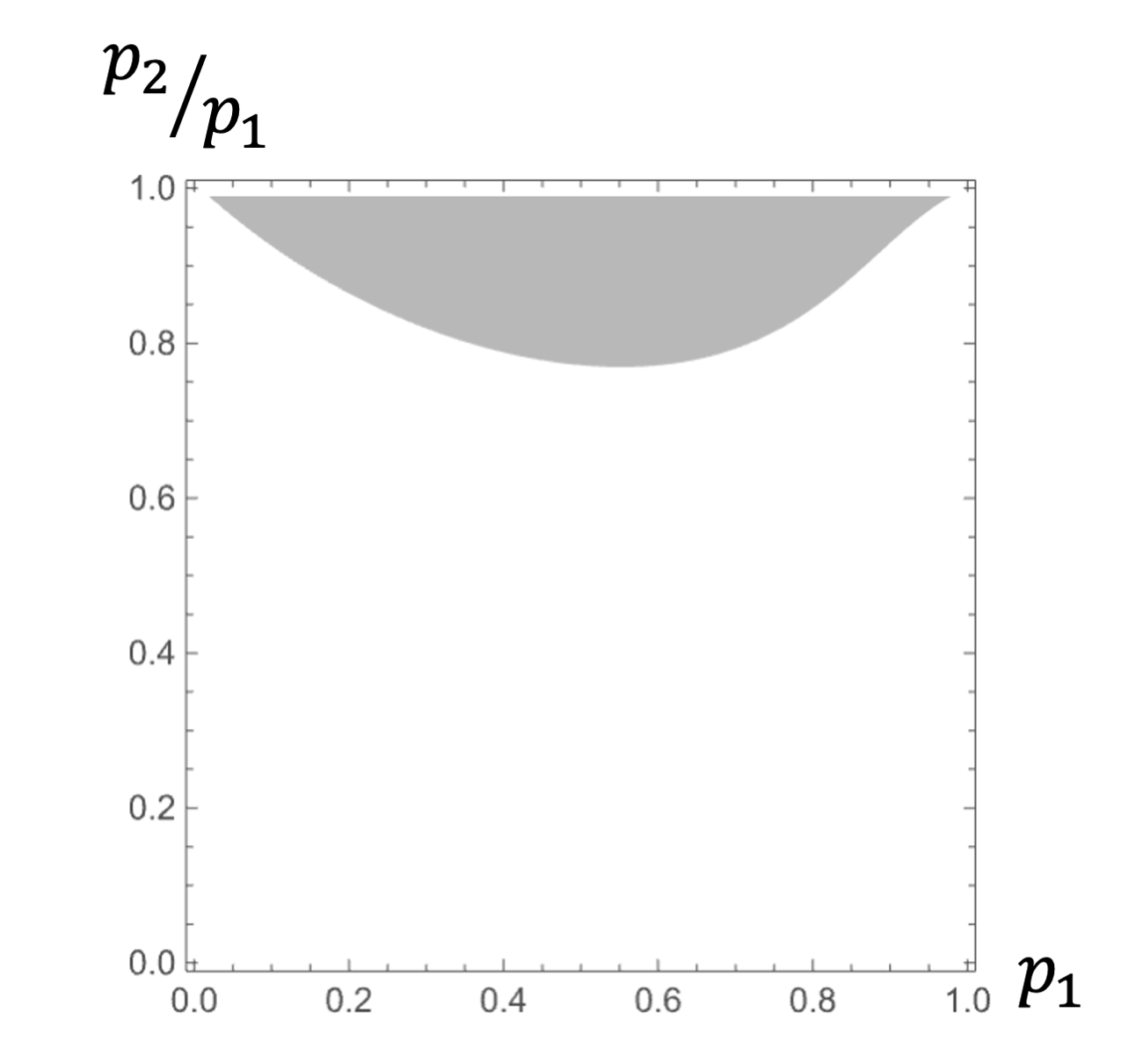}
        \caption{$b_A/c=5$}
        \label{Fig:asymmetric_information_5}
    \end{subfigure}

    \caption{Participation region from Firm 1's perspective in the 
    $(p_1,p_2/p_1)$ parameter space, $\omega=0.2$.}
    \label{Fig:asymmetric_information}
\end{figure}

The participation regions in Figure \ref{Fig:asymmetric_information} suggest that the high-information power firm (Firm 1) is more likely to participate as the information-power gap decreases (i.e., as $p_2$ increases). 
As market sensitivity decreases with respect to production cost, Firm 1 tolerates larger information gaps in its participation decision.

Overall, both forms of heterogeneity create different participation bottlenecks: cost-inefficient firms under cost asymmetry, high-information power firms under information power asymmetry. 
Thus, the effect of heterogeneity on participation depends on both the source and degree of the asymmetry.

\section{Numerical Study: Networked Competition Structure Effects on Participation}
\label{Sec:Numerical}

The two-firm analysis examines settings in which firms compete directly and explores how firm and market characteristics affect platform formation.
When more than two firms interact across multiple markets, firms are subject to both direct and indirect competition, and participation decisions depend on the broader network competition structure, which is determined by the topology of the network connecting firms to markets.
Thus, in this section, we conduct a numerical study of supply chains comprising three firms and four markets, varying the network topology in addition to the other factors captured in Section \ref{Sec:Analysis}.  Our goal is to examine whether observations in the direct-competition environment extend to more complex supply chain network competition structures, and to identify how changes in supply chain competition density and firms’ exposure to local competition affect firms' participation decisions in the presence of varying information power and shock magnitude.

In the numerical study, we set the baseline market size to $a_j=a=1000$ for each Market $j$ and the shock probability to $\omega=0.1$. We assume a uniform distribution of shock locations across markets, with $r_j=r=0.25$ for all $j$.
Other parameter variants are created by modifying this baseline setting and are described in the subsequent sections. 

\subsection{\emph{
Role of Competition Network on Platform Formation
}}\label{Sec:network_effect}

We first examine how network competition structure affects platform formation under different information-power and shock conditions. 
To do so, we use five networks with distinct topologies, as shown in Figure \ref{Fig:networks}, that differ in (i) \textit{overall intensity of competition}, determined by the network density, and (ii) \textit{firms' local competition exposure}, determined by the allocation of firm-market pairs.

\begin{figure}[h]
\centering
\includegraphics[width=0.8\textwidth]{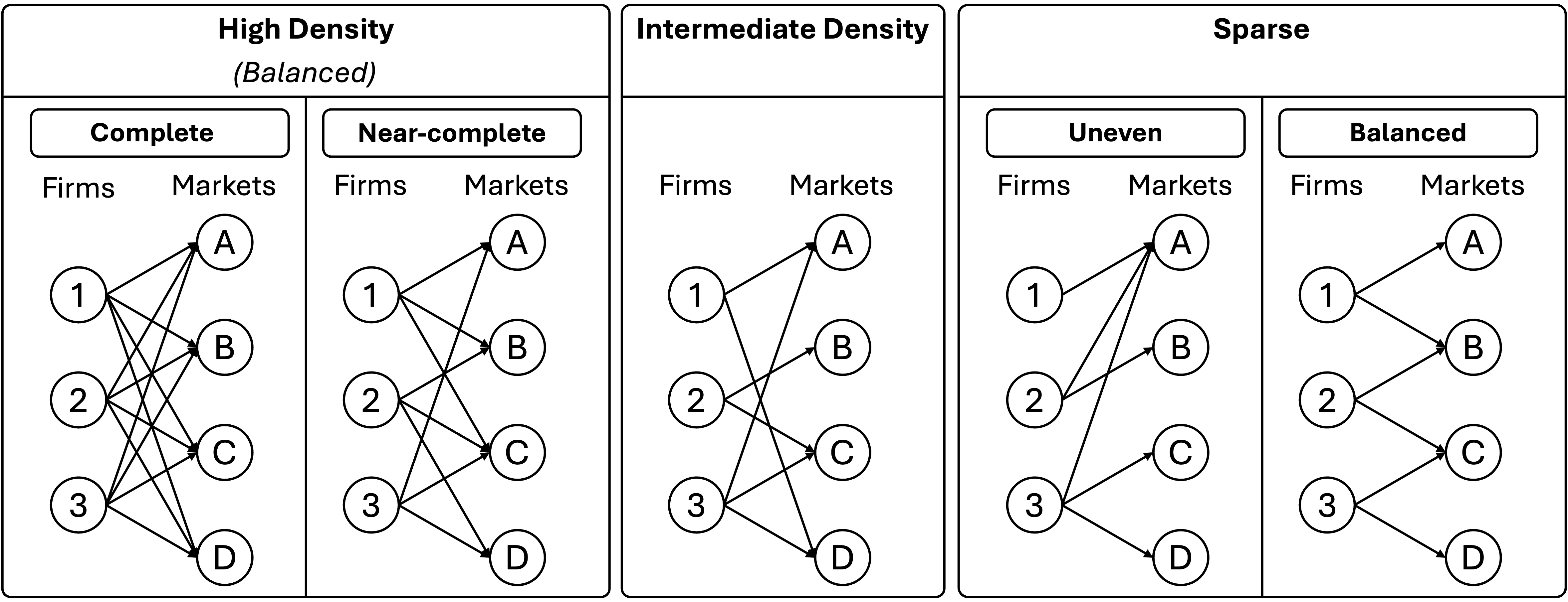}
\caption{Supply chain network topologies representing different competition structures.}
\label{Fig:networks}
\end{figure}
To isolate the effects of network topology from firm and market heterogeneity, for each firm, we fix production costs at $c_i=c=1$ and set market sensitivity at $b_j=b=1$ for each market. 
For firms, we consider three levels of information power: high ($p=0.75$), medium ($p=0.5$), and low ($p=0.25$). We examine all combinations corresponding to two scenarios: (a) asymmetric information power, where one firm has higher information power than the other two homogeneous firms and (b) asymmetric information power across firms.
For disruptions, we vary the normalized shock size, $\frac{D}{a}$, across three levels: low ($\frac{D}{a}=0.25$), medium ($\frac{D}{a}=1$), and high ($\frac{D}{a}=2$). We then employ a full-factorial experimental design across all parameter combinations.

\noindent \textbf{Firms with Asymmetric Information Power:}
 Earlier in Section \ref{Sec:Analysis}, we observed that when firms differ only in information power, the one with high information power has less incentive to participate because it gives up relatively more of its informational advantage while receiving less in return. 
We observe a similar pattern across network topologies, with participation reluctance strongest in the complete network, where firms face the greatest overall competitive exposure.
In Table~\ref{tab:network_total_density_effect}, we see that the firm with higher information (i.e., Firm 1) never participates in the platform when the network is fully connected, but does so when the supply network becomes less connected. 
Relative to the complete network, reducing total network density, while keeping the distribution of local competition balanced, results in full participation in several instances. 
This is because, as the supply chain network becomes less connected, the better-informed firm is less negatively affected by the increase in competitors’ information advantage resulting from the information it shares.

\begin{table}[h]
\centering
\caption{Equilibrium platform profiles under different network densities.} 
\label{tab:network_total_density_effect}

\begin{threeparttable}

\renewcommand{\arraystretch}{1.15}
\setlength{\tabcolsep}{11pt}

\resizebox{0.7\textwidth}{!}{%
\begin{tabular}{c|ccc||ccc}
\toprule
\textbf{Information power}
& \multicolumn{3}{c||}{\textbf{Complete Network}}
& \multicolumn{3}{c}{\textbf{Near-compelete Network}}
\\

\cmidrule(lr){2-4}
\cmidrule(lr){5-7}

$(p_1,p_2,p_3)$ &
\multicolumn{3}{c||}{Shock Magnitude} & \multicolumn{3}{c}{Shock Magnitude} \\
& \textbf{Low}
& \textbf{Medium}
& \textbf{High}
& \textbf{Low}
& \textbf{Medium}
& \textbf{High}
\\

\midrule

$(0.75,0.50,0.50)$
& S & S & S
& S & A & A
\\

\midrule

$(0.50,0.25,0.25)$
& -- & S & S
& S & S & S
\\

\midrule

$(0.75,0.25,0.25)$
& S & S & S
& S & S & S
\\

\bottomrule

\end{tabular}
}

\begin{tablenotes}[flushleft]
\footnotesize
\item \textit{Notes:}
A = all firms join;
S = only information-symmetric firms join;
-- = no platform equilibrium observed.
\end{tablenotes}

\end{threeparttable}
\end{table}

Participation incentives also depend on how competitive exposure is distributed among firms. Table \ref{tab:network_sparse_density} compares two sparse network instances with the same total density but different distributions of local competition. 
The high-information power firm is reluctant to participate when it faces greater local competition, but participates more often when it faces less local competition. For example, in the uneven sparse network, when Firm 1 has high information power, it does not participate under medium and high shocks because other firms also serve the only market Firm 1 serves, creating direct competition in that market. 
Conversely, in the balanced sparse network, the platform attracts more firms compared to the uneven network across all parameter settings.

\begin{table}[h]
\centering
\caption{Equilibrium platform profiles under different local competitive exposure.}
\label{tab:network_sparse_density}

\begin{threeparttable}

\renewcommand{\arraystretch}{1.15}
\setlength{\tabcolsep}{10pt}

\resizebox{0.7\textwidth}{!}{%
\begin{tabular}{c|ccc||ccc}
\toprule
\textbf{Information power}
& \multicolumn{3}{c||}{\textbf{Sparse--Uneven}}
& \multicolumn{3}{c}{\textbf{Sparse--Balanced}}
\\

\cmidrule(lr){2-4}
\cmidrule(lr){5-7}

$(p_1,p_2,p_3)$ &
\multicolumn{3}{c||}{Shock Magnitude} & \multicolumn{3}{c}{Shock Magnitude} \\
& \textbf{Low} & \textbf{Medium} & \textbf{High} & \textbf{Low} & \textbf{Medium} & \textbf{High} \\

\midrule

$(0.75,0.50,0.50)$ & A & S & S & A & A & A \\
$(0.50,0.75,0.50)$ & A & A & A & A & A & A \\
$(0.50,0.50,0.75)$ & A & A & A & A & A & A \\

\midrule

$(0.50,0.25,0.25)$ & A & S & S & A & A & A \\

$(0.25,0.50,0.25)$ & A & A & A & A & A & A \\

$(0.25,0.25,0.50)$ & A & A & A & A & A & A \\

\midrule

$(0.75,0.25,0.25)$ & S & S & S & S & S & A \\

$(0.25,0.75,0.25)$ & S & -- & -- & S & S & S \\

$(0.25,0.25,0.75)$ & S & S & S & S &S & A \\

\bottomrule

\end{tabular}
}

\begin{tablenotes}[flushleft]
\footnotesize
\item \textit{Notes:}
A = all firms join;
S = only information-symmetric firms join;
-- = no platform equilibrium observed;
\end{tablenotes}
\end{threeparttable}
\end{table}

\noindent \textbf{Firms with Symmetric Information Power:}
When firms have equal information power, full participation is substantially more common than under asymmetric information power, and participation is largely insensitive to network topology. 
Across all settings with symmetric information power, full participation occurs except in one case: the complete network under the smallest shock and lowest information power ($\frac{D}{a}=0.25$ and $p_i=p=0.25$). 
With symmetric information power, neither firm has an informational advantage that deters participation, while information sharing allows both firms to reduce uncertainty about their competitors' information and decisions. 
Because full participation occurs in all but one case, we do not report these results in a separate table.

\subsection{\emph{
System Parameter Effects on Platform Formation Across Networks
}}
\label{Sec:production_cost_effect}

Next, we investigate how changes in firms' production costs and market sensitivity affect equilibrium platform participation across varying network structures and shock magnitude, as these parameters are the most influential in the platform characteristics.
To do so, we compare participation decisions in the \emph{`Complete'} network, representing a high level of competition, with those in the \emph{`Intermediate-density'} network, representing a lower level of competition. We vary the normalized shock magnitude to represent different levels of disruption intensity: low (0.5), medium (1), high (1.5).
We set $c_i=1$ and $b=1$ as their respective base levels and vary one parameter at a time while holding the other at its base level. 
For the selected parameter, we consider progressively lower and higher values from the set $\{0.05, 0.1, 0.25, 0.5, 1.5, 2\}$.

\begin{table}[h]
\centering
\caption{Effect of production cost changes on equilibrium platform participation.}
\label{tab:cost_change_participation}

\begin{threeparttable}

\renewcommand{\arraystretch}{1.15}
\setlength{\tabcolsep}{9pt}

\resizebox{0.99\textwidth}{!}{%
\begin{tabular}{ccc|ccc||ccc}
\toprule

& & &
\multicolumn{6}{c}{\textbf{Cost change for the considered firm}}
\\

\cmidrule(lr){4-9}
\textbf{Considered Firm} &
\textbf{Network}
& \textbf{Information power}
& \multicolumn{3}{c||}{\textbf{Increasing cost} }
& \multicolumn{3}{c}{\textbf{Decreasing cost} }
\\

\cmidrule(lr){4-6}
\cmidrule(lr){7-9}

&& $(p_1,p_2,p_3)$ 
&\multicolumn{3}{c||}{Shock Magnitude $\frac{D}{a}$} & \multicolumn{3}{c}{Shock Magnitude $\frac{D}{a}$}
\\

&&  
& \textbf{0.5}
& \textbf{1}
& \textbf{1.5}
& \textbf{0.5}
& \textbf{1}
& \textbf{1.5}
\\

\midrule
\multirow{2}{*}{\textbf{Firm 1}} &
\multirow{2}{*}{\textbf{Complete}} & $(0.25,0.25,0.25)$ & A & A & A & $A/ -^{\dagger}$ & $A/-^{\dagger}$ & A \\

&& $(0.75,0.25,0.25)$ & -- & C & C & C & C & C \\

\midrule
\multirow{2}{*}{\textbf{Firm 2}} &
\multirow{2}{*}{\textbf{Intermediate}} & $(0.25,0.25,0.25)$ & A & A & A & A & A & A \\

&& $(0.25,0.75,0.25)$ & C & C & C & C & C & C \\

\midrule
\multirow{2}{*}{\textbf{Firm 3}} &
\multirow{2}{*}{\textbf{Intermediate}} & $(0.25,0.25,0.25)$ & A & A & A & A & A & A \\

&& $(0.25,0.25,0.75)$ & -- & -- & C & C & C & C \\

\bottomrule

\end{tabular}%
}

\end{threeparttable}

\vspace{2pt}

\parbox{\textwidth}{%
\footnotesize
\textit{Notes:}
The selected firm is the firm whose production cost is varied in the numerical studies; 
Starting from base level $c_i=1$, selected firm's cost increased to $\{1.5,2\}$ or decreased to $\{0.5,0.25,0.1,0.05\}$;
A = all firms join;
C = only cost-symmetric firms join;
-- = no platform equilibrium is observed;
$A/-^{\dagger}$ = initially all firms participate and when the production cost is sufficiently low, the platform no longer forms. 
}

\end{table}

\noindent{\textbf{Production Cost Effect:}}
In the production-cost analysis, we change the cost of one firm at a time. This firm is marked as `Considered Firm' in Table \ref{tab:cost_change_participation} (Firm 1 in \emph{Complete} and Firm 2 or Firm 3 in \emph{Intermediate-density}).
The table summarizes the equilibrium platform configurations as we change the selected firm's production cost. 
Platform configurations are categorized according to whether 
only the two cost-symmetric firms join (denoted by $C$ in the table) or whether all three firms join (denoted by $A$ in the table).

The effect of production cost asymmetry on platform participation depends on how the asymmetry is distributed. When one firm is sufficiently more cost-efficient than its competitors, the inefficient firms may be discouraged from sharing information with the cost-efficient firm and the platform does not form. 
On the contrary, if a single firm is cost-inefficient compared to its competitors, it chooses to participate in the platform to avoid missing out on the collective information. 
These observations suggest that when the cost-inefficient firms are the majority, they are 
reluctant to form a platform with a more efficient competitor.
However, when they are the minority, they
prefer to join when their more efficient competitors form a platform and collectively reach a critical mass of shared information.

When one firm has greater information power than its competitors and becomes more cost-efficient (i.e., decreasing cost), its competitors, which previously remained outside the platform, begin to participate and share information with one another, while the more efficient and better-informed firm remains outside the platform in equilibrium.
The reason is that the firm in consideration now has advantages in both information and production efficiency, making it a particularly strong competitor. In response, the other firms benefit from coordinating through the platform and sharing information with each other, thereby strengthening their competitive position against the advantaged firm.

 \noindent{\textbf{Market Sensitivity Effect:}} Table \ref{Tab:SensitivityEffect} shows the market sensitivity effect on firms' participation decisions across different supply chain networks and information-power gaps between firms. 
Table \ref{Tab:SensitivityEffect} shows how market sensitivity affects firms' participation decisions across different network structures and information-power configurations. Consistent with the analysis in Section \ref{Sec:Analysis}, decreasing market sensitivity generally increases participation across all considered supply chain networks. 
Market sensitivity has a stronger effect when information power is asymmetric. At the base or higher levels of market sensitivity, the information-disadvantaged firms often participate and share information with one another, while the information-advantaged firm remains outside the platform. 
This is an example where low-information firms join forces against the high-information firm to become more competitive.
As market sensitivity decreases, the competition becomes less fierce, and the equilibrium shifts toward full participation, with the information-advantaged firm also joining the platform. Thus, lower market sensitivity primarily promotes participation by firms that would otherwise remain outside the platform because of their information advantage.

\begin{table}[h]
\centering
\caption{Equilibrium participation configurations under market sensitivity effect.}
\label{tab:sensitivity_change_participation}

\resizebox{0.75\textwidth}{!}{%
\begin{tabular}{llcccc}
\hline
\textbf{Network} & \textbf{Information power}  & \textbf{Market sensitivity} & \multicolumn{3}{c}{\textbf{Shock magnitude }$\frac{D}{a}$} \\
\cline{4-6}
& $(p_1,p_2,p_3)$  & & 0.5 & 1 & 1.5 \\
\hline
\multirow{6}{*}{Complete} & \multirow{3}{*}{$(0.25,0.25,0.25)$} & $\uparrow$ & $A$ & $A$ & $A$ \\
& & base case ($b=1$) & $A$ & $A$ & $A$ \\
& & $\downarrow$ & $A$ & $A$ & $A$ \\
\cline{2-6}
& \multirow{3}{*}{$(0.75,0.25,0.25)$} & $\uparrow$ & $-$ & $S$ & $S$ \\
& & base case ($b=1$) & $S$ & $S$ & $S$ \\
& & $\downarrow$ & $S/A^\dagger$ & $S/A^\dagger$ & $A$ \\
\hline
\multirow{9}{*}{Intermediate} & \multirow{3}{*}{$(0.25,0.25,0.25)$} & $\uparrow$ & $A$ & $A$ & $A$ \\
& & base case ($b=1$) & $A$ & $A$ & $A$ \\
& & $\downarrow$ & $A$ & $A$ & $A$ \\
\cline{2-6}
& \multirow{3}{*}{$(0.25,0.75,0.25)$} & $\uparrow$ & $S$ & $S$ & $S$ \\
& & base case ($b=1$) & $S$ & $S$ & $S$ \\
& & $\downarrow$ & $A$ & $A$ & $A$ \\
\cline{2-6}
& \multirow{3}{*}{$(0.25,0.25,0.75)$} & $\uparrow$ & $-$ & $S$ & $S$ \\
& & base case ($b=1$) & $-$ & $S$ & $S$ \\
& & $\downarrow$ & $A$ & $A$ & $A$ \\
\hline
\end{tabular}
}

\vspace{2pt}

\begin{minipage}{0.95\textwidth}
\footnotesize
\textit{Notes:} $A$ = all firms join; $S$ = only information-symmetric firms join; 
$-$ = no platform equilibrium is observed. $S/A^\dagger$ indicates a transitional case: for a moderate decrease in market sensitivity, the equilibrium is ``information-symmetric firms join,'' whereas for a sufficiently large decrease, it shifts to ``all firms join.'' 
\end{minipage}
\label{Tab:SensitivityEffect}
\end{table}

\section{Summary}\label{Sec:Conc}

In this paper, we study the participation of horizontally competing supply chain firms in a centralized, multi-firm information-sharing platform to share  market shock signals that affect demand. 
We consider firms with different information capabilities and the participation problem prior to the realization of private signals, thereby creating a stochastic information-asymmetry environment.  The key question for each firm is whether the benefits of information sharing outweigh the loss of competitive advantage from sharing their private information. 
We develop a game-theoretic model to identify the supply chain characteristics that shape firms' incentives to participate in the platform.
We first analyze the equilibrium participation strategies when there are two competing firms, providing closed-form results for equilibrium participation behavior that highlight the effect of information power and production efficiency on firms' participation decisions. Numerical results, capturing more complex competition network structures, show that both network structure and the position of each firm in the network critically influence participation incentives.

Our findings reveal how supply chain characteristics impact participation in information sharing. 
First, the level of information asymmetry is a key factor in the equilibrium platform profile. Firms with similar information power may benefit from joining forces to reduce their competitive disadvantage against higher information power competitors.
Second, cost-inefficient firms may be less willing to participate because their competitors can exploit shared information, and uncertainty can serve as a competitive shield by restricting the actions of more efficient competitors. 
However, when several competitors join the platform, a cost-inefficient firm may also choose to participate to avoid an informational disadvantage.
Third, market sensitivity affects the competitive role of information and uncertainty. In highly sensitive markets, where firms have less flexibility to adjust their production quantities, information becomes more valuable. At the same time, uncertainty becomes more challenging for cost-inefficient firms. Consequently, greater market sensitivity increases the reluctance of high-information power firms to participate, while cost-inefficient firms are more willing to participate.
Lastly, we analyze how supply chain network topology shapes information sharing incentives among competing firms. 
Although competing firms are generally considered to be reluctant to share information, our results show that they may do so when information sharing creates an opportunity to redistribute their efforts between the markets they serve. Specifically, strategic information sharing may reduce the intensity in shared markets by enabling firms to focus on the markets where they have stronger positions.

There are a number of directions open for future research. First, this work assumes the signals being shared are truthful, and future research could explore mechanisms that penalize untruthful reporting.  Further, different disruption settings and types of information to share could be considered. In addition, future research could extend the single-period setting to multiple operational periods with a defined participation horizon that captures how updates in beliefs affect participation.
Another direction worth exploring is how to design platform mechanisms to increase participation.
For example, firms could be compensated to participate by creating reward allocation schemes funded by participation fees collected from firms benefiting more from the platform. Finally, in this work, we assume that firms make their decisions based on financial gains alone, omitting privacy concerns, to isolate the impact of supply chain conditions on platform benefits, but future work could explore privacy-preserving methods to encourage information sharing in competing supply chains.
\vspace{-0.5cm}

\subsection*{Data Availability Statement}
The data supporting the findings of this study are available within the article.
\vspace{-0.5cm}

\subsection*{Declaration of Generative AI Use}

The authors used OpenAI ChatGPT for minor language refinement.

\bibliographystyle{apalike}
\spacingset{1}
\bibliography{reference}

@article{ZHONG2023,
title = {The incentives for information sharing in online retail platforms},
author = {Qinjia Zhong and Jianjun Wang and Zongbao Zou and Xiaofan Lai},
journal = {Transportation Research Part E: Logistics and Transportation Review},
volume = {172},
pages = {103050},
year = {2023},
issn = {1366-5545},
doi = {https://doi.org/10.1016/j.tre.2023.103050},
url = {https://www.sciencedirect.com/science/article/pii/S1366554523000388}}

@article{zha2023strategic,
  title={Strategic information sharing of online platforms as resellers or marketplaces},
  author={Zha, Yong and Li, Quan and Huang, Tingliang and Yu, Yugang},
  journal={Marketing Science},
  volume={42},
  number={4},
  pages={659--678},
  year={2023},
  publisher={INFORMS}
}

@article{Li2021online,
author = {Guo Li and Lin Tian and Hong Zheng},
title ={Information Sharing in an Online Marketplace with Co‐opetitive Sellers},
journal = {Production and Operations Management},
volume = {30},
number = {10},
pages = {3713-3734},
year = {2021},
doi = {10.1111/poms.13460},
URL = { 
    
        https://doi.org/10.1111/poms.13460

},
eprint = { 
    
        https://doi.org/10.1111/poms.13460
}
}

@article{zhang2006competition,
  title={Competition, cooperation, and information sharing in a two-echelon assembly system},
  author={Zhang, Fuqiang},
  journal={Manufacturing \& Service Operations Management},
  volume={8},
  number={3},
  pages={273--291},
  year={2006},
  publisher={INFORMS}
}

@article{shang2016information,
  title={Information sharing in a supply chain with a common retailer},
  author={Shang, Weixin and Ha, Albert Y and Tong, Shilu},
  journal={Management Science},
  volume={62},
  number={1},
  pages={245--263},
  year={2016},
  publisher={INFORMS}
}

@ARTICLE{Wang2020,
  author={Wang, Jian-Jun and Dong, Jingyi and Yue, Xiaohang and Zhong, Qinjia},
  journal={IEEE Transactions on Systems, Man, and Cybernetics: Systems}, 
  title={Information Sharing in a Supply Chain With a Coopetitive Contract Manufacturer}, 
  year={2020},
  volume={50},
  number={5},
  pages={1777-1793},
  doi={10.1109/TSMC.2017.2786867}}

@article{Hou28012026,
author = {Pengwen Hou and Zihan Liu and Hubert Pun and Jun Wang},
title = {Platform’s information sharing strategy in a supplier encroachment context with economies of scale},
journal = {IISE Transactions},
volume = {0},
number = {0},
pages = {1--16},
year = {2026},
publisher = {Taylor \& Francis},
doi = {10.1080/24725854.2026.2622597},


URL = { 
    
        https://doi.org/10.1080/24725854.2026.2622597
    
    

},
eprint = { 
    
        https://doi.org/10.1080/24725854.2026.2622597
    
    

}

}

@article{liu2021information,
  title={Information sharing on retail platforms},
  author={Liu, Zekun and Zhang, Dennis J and Zhang, Fuqiang},
  journal={Manufacturing \& Service Operations Management},
  volume={23},
  number={3},
  pages={606--619},
  year={2021},
  publisher={INFORMS}
}

@article{Eckstein19052015,
author = {Dominik Eckstein and Matthias Goellner and Constantin Blome and Michael Henke},
title = {The performance impact of supply chain agility and supply chain adaptability: the moderating effect of product complexity},
journal = {International Journal of Production Research},
volume = {53},
number = {10},
pages = {3028--3046},
year = {2015},
publisher = {Taylor \& Francis},
doi = {10.1080/00207543.2014.970707},
URL = {    
        https://doi.org/10.1080/00207543.2014.970707
},
eprint = {     
        https://doi.org/10.1080/00207543.2014.970707
}
}

@article{cachon2000supply,
  title={Supply chain inventory management and the value of shared information},
  author={Cachon, Gerard P and Fisher, Marshall},
  journal={Management Science},
  volume={46},
  number={8},
  pages={1032--1048},
  year={2000},
  publisher={Informs}
}

@article{VIVES198471,
title = {Duopoly information equilibrium: Cournot and bertrand},
journal = {Journal of Economic Theory},
volume = {34},
number = {1},
pages = {71-94},
year = {1984},
issn = {0022-0531},
doi = {https://doi.org/10.1016/0022-0531(84)90162-5},
url = {https://www.sciencedirect.com/science/article/pii/0022053184901625},
author = {Xavier Vives}
}

@article{li2002,
author = {Li, Lode},
title = {Information Sharing in a Supply Chain with Horizontal Competition},
journal = {Management Science},
volume = {48},
number = {9},
pages = {1196-1212},
year = {2002},
doi = {10.1287/mnsc.48.9.1196.177},

URL = { 
    
        https://doi.org/10.1287/mnsc.48.9.1196.177
    
    

},
eprint = { 
    
        https://doi.org/10.1287/mnsc.48.9.1196.177
    
    

}}

@article{zhang2002inference,
author = {Hongtao Zhang},
title ={VERTICAL INFORMATION EXCHANGE IN A SUPPLY CHAIN WITH DUOPOLY RETAILERS},
journal = {Production and Operations Management},
volume = {11},
number = {4},
pages = {531-546},
year = {2002},
doi = {10.1111/j.1937-5956.2002.tb00476.x},
URL = {     
        https://doi.org/10.1111/j.1937-5956.2002.tb00476.x
},
eprint = {    
        https://doi.org/10.1111/j.1937-5956.2002.tb00476.x

}
}

@article{bimpikis2019cournot,
  title={Cournot competition in networked markets},
  author={Bimpikis, Kostas and Ehsani, Shayan and Ilk{\i}l{\i}{\c{c}}, Rahmi},
  journal={Management Science},
  volume={65},
  number={6},
  pages={2467--2481},
  year={2019},
  publisher={INFORMS}
}

@article{choi2008value,
  title={Value of sharing production yield information in a serial supply chain},
  author={Choi, Hyun-cheol Paul and Blocher, James D and Gavirneni, Srinagesh},
  journal={Production and Operations Management},
  volume={17},
  number={6},
  pages={614--625},
  year={2008},
  publisher={SAGE Publications Sage CA: Los Angeles, CA}
}

@article{CHO201397,
title = {The value of information sharing in a supply chain with a seasonal demand process},
journal = {Computers \& Industrial Engineering},
volume = {65},
number = {1},
pages = {97-108},
year = {2013},
note = {Intelligent Manufacturing Systems},
issn = {0360-8352},
doi = {https://doi.org/10.1016/j.cie.2011.12.004},
url = {https://www.sciencedirect.com/science/article/pii/S0360835211003779},
author = {Dong Won Cho and Young Hae Lee}
}

@article{lee1997information,
  title={Information distortion in a supply chain: The bullwhip effect},
  author={Lee, Hau L and Padmanabhan, Venkata and Whang, Seungjin},
  journal={Management Science},
  volume={43},
  number={4},
  pages={546--558},
  year={1997},
  publisher={Informs}
}

@ARTICLE{linhuanglin1046059,
  author={Furen Lin and Shenghsiu Huang and Sheng-cheng Lin},
  journal={IEEE Transactions on Engineering Management}, 
  title={Effects of information sharing on supply chain performance in electronic commerce}, 
  year={2002},
  volume={49},
  number={3},
  pages={258-268},
  doi={10.1109/TEM.2002.803388}}

@article{ZHAO2002321,
title = {The impact of forecasting model selection on the value of information sharing in a supply chain},
journal = {European Journal of Operational Research},
volume = {142},
number = {2},
pages = {321-344},
year = {2002},
issn = {0377-2217},
doi = {https://doi.org/10.1016/S0377-2217(01)00300-9},
url = {https://www.sciencedirect.com/science/article/pii/S0377221701003009},
author = {Xiande Zhao and Jinxing Xie and Janny Leung}
}

@article{chen2000quantifying,
  title={Quantifying the bullwhip effect in a simple supply chain: The impact of forecasting, lead times, and information},
  author={Chen, Frank and Drezner, Zvi and Ryan, Jennifer K and Simchi-Levi, David},
  journal={Management Science},
  volume={46},
  number={3},
  pages={436--443},
  year={2000},
  publisher={INFORMS}
}

@article{wang2009markdown,
  title={Markdown money contracts for perishable goods with clearance pricing},
  author={Wang, Charles X and Webster, Scott},
  journal={European Journal of Operational Research},
  volume={196},
  number={3},
  pages={1113--1122},
  year={2009},
  publisher={Elsevier}
}

@article{cachon2011value,
  title={The value of fast fashion: Quick response, enhanced design, and strategic consumer behavior},
  author={Cachon, G{\'e}rard P and Swinney, Robert},
  journal={Management Science},
  volume={57},
  number={4},
  pages={778--795},
  year={2011},
  publisher={INFORMS}
}

@article{CAO2013425,
title = {Coordination of a supply chain with one manufacturer and multiple competing retailers under simultaneous demand and cost disruptions},
journal = {International Journal of Production Economics},
volume = {141},
number = {1},
pages = {425-433},
year = {2013},
note = {Meta-heuristics for manufacturing scheduling and logistics problems},
issn = {0925-5273},
doi = {https://doi.org/10.1016/j.ijpe.2012.09.009},
url = {https://www.sciencedirect.com/science/article/pii/S0925527312004057},
author = {Erbao Cao and Can Wan and Mingyong Lai}}

@article{wu2021anti,
  title={The anti-collusion dilemma: Information sharing of the supply chain under buyback contracts},
  author={Wu, Jiang and Zou, Liuxin and Gong, Yeming and Chen, Mingyang},
  journal={Transportation Research Part E: Logistics and Transportation Review},
  volume={152},
  pages={102413},
  year={2021},
  publisher={Elsevier}
}

@article{lee2000value,
  title={The value of information sharing in a two-level supply chain},
  author={Lee, Hau L and So, Kut C and Tang, Christopher S},
  journal={Management Science},
  volume={46},
  number={5},
  pages={626--643},
  year={2000},
  publisher={INFORMS}
}

@article{zhang2013coordination,
  title={Coordination of information sharing in a supply chain},
  author={Zhang, Juliang and Chen, Jian},
  journal={International Journal of Production Economics},
  volume={143},
  number={1},
  pages={178--187},
  year={2013},
  publisher={Elsevier}
}

@article{shen2019review,
  title={A review on supply chain contracting with information considerations: information updating and information asymmetry},
  author={Shen, Bin and Choi, Tsan-Ming and Minner, Stefan},
  journal={International Journal of Production Research},
  volume={57},
  number={15-16},
  pages={4898--4936},
  year={2019},
  publisher={Taylor \& Francis}
}

@article{RAWEEWAN2018269,
title = {Information sharing in supply chain collaboration},
journal = {Computers \& Industrial Engineering},
volume = {126},
pages = {269-281},
year = {2018},
issn = {0360-8352},
doi = {https://doi.org/10.1016/j.cie.2018.09.042},
url = {https://www.sciencedirect.com/science/article/pii/S0360835218304583},
author = {Morrakot Raweewan and William G. Ferrell}}

@article{shin2010firms,
  title={Do firms invest in forecasting efficiently? The effect of competition on demand forecast investments and supply chain coordination},
  author={Shin, Hyoduk and Tunca, Tunay I},
  journal={Operations Research},
  volume={58},
  number={6},
  pages={1592--1610},
  year={2010},
  publisher={INFORMS}
}

@article{inference,
author = {Cao, Erbao and Chen, Guangzhi},
title = {Information sharing motivated by production cost reduction in a supply chain with downstream competition},
journal = {Naval Research Logistics (NRL)},
volume = {68},
number = {7},
pages = {898-907},
doi = {https://doi.org/10.1002/nav.21977},
url = {https://onlinelibrary.wiley.com/doi/abs/10.1002/nav.21977},
year = {2021}}

@article{hawithcostreduc,
author = {Ha, Albert Y. and Tian, Quan and Tong, Shilu},
title = {Information Sharing in Competing Supply Chains with Production Cost Reduction},
journal = {Manufacturing \& Service Operations Management},
volume = {19},
number = {2},
pages = {246-262},
year = {2017},
doi = {10.1287/msom.2016.0607}}

@article{ha2008contracting,
  title={Contracting and information sharing under supply chain competition},
  author={Ha, Albert Y and Tong, Shilu},
  journal={Management Science},
  volume={54},
  number={4},
  pages={701--715},
  year={2008},
  publisher={Informs}
}

@article{ZHOU2017785,
title = {Supply chain coordination with information sharing: The informational advantage of GPOs},
journal = {European Journal of Operational Research},
volume = {256},
number = {3},
pages = {785-802},
year = {2017},
issn = {0377-2217},
doi = {https://doi.org/10.1016/j.ejor.2016.06.045},
url = {https://www.sciencedirect.com/science/article/pii/S0377221716304854},
author = {Maosen Zhou and Bin Dan and Songxuan Ma and Xumei Zhang}
}

@inproceedings{yuan2008incentive,
  title={Incentive Mechanism of Supply Chain Inventory Management Based on Information Sharing},
  author={Yuan, Qi and Qiong, Zhang},
  booktitle={2008 International Seminar on Business and Information Management},
  volume={2},
  pages={373--377},
  year={2008},
  organization={IEEE}
}

@incollection{li2002supply,
  author    = {Li, Lode and Zhang, Hongtao},
  editor    = {Song, Jing-Sheng and Yao, David D.},
  title     = {Supply Chain Information Sharing in a Competitive Environment},
  booktitle = {Supply Chain Structures: Coordination, Information and Optimization},
  year      = {2002},
  publisher = {Springer US},
  address   = {Boston, MA},
  pages     = {161--206},
  doi       = {10.1007/978-1-4757-6635-6_6}
}

@article{Hafirstpaper,
author = {Ha, Albert Y.},
title = {Supplier-buyer contracting: Asymmetric cost information and cutoff level policy for buyer participation},
journal = {Naval Research Logistics (NRL)},
volume = {48},
number = {1},
pages = {41-64},
doi = {https://doi.org/10.1002/1520-6750(200102)48:1<41::AID-NAV3>3.0.CO;2-M},
url = {https://onlinelibrary.wiley.com/doi/abs/10.1002/1520-6750%28200102%2948%3A1%3C41%3A%3AAID-NAV3%3E3.0.CO%3B2-M},
eprint = {https://onlinelibrary.wiley.com/doi/pdf/10.1002/1520-6750%28200102%2948%3A1%3C41%3A%3AAID-NAV3%3E3.0.CO%3B2-M},
year = {2001}
}

@article{MASSARI2021108150,
title = {Investigating the effect of horizontal coopetition on supply chain resilience in complex and turbulent environments},
journal = {International Journal of Production Economics},
volume = {237},
pages = {108150},
year = {2021},
issn = {0925-5273},
doi = {https://doi.org/10.1016/j.ijpe.2021.108150},
url = {https://www.sciencedirect.com/science/article/pii/S0925527321001262},
author = {Giovanni Francesco Massari and Ilaria Giannoccaro}
}

@article{fang2015managing,
  title={Managing supply uncertainty under supply chain Cournot competition},
  author={Fang, Yaner and Shou, Biying},
  journal={European Journal of Operational Research},
  volume={243},
  number={1},
  pages={156--176},
  year={2015},
  publisher={Elsevier}
}

@article{chen2019supply,
  title={Supply chain contracts that prevent information leakage},
  author={Chen, Yiwei and {\"O}zer, {\"O}zalp},
  journal={Management Science},
  volume={65},
  number={12},
  pages={5619--5650},
  year={2019},
  publisher={INFORMS}
}

@article{costantino2014impact,
  title={The impact of information sharing and inventory control coordination on supply chain performances},
  author={Costantino, Francesco and Di Gravio, Giulio and Shaban, Ahmed and Tronci, Massimo},
  journal={Computers \& Industrial Engineering},
  volume={76},
  pages={292--306},
  year={2014},
  publisher={Elsevier}
}

@article{huo2014effects,
  title={The effects of competitive environment on supply chain information sharing and performance: an empirical study in China},
  author={Huo, Baofeng and Zhao, Xiande and Zhou, Honggeng},
  journal={Production and Operations Management},
  volume={23},
  number={4},
  pages={552--569},
  year={2014},
  publisher={Wiley Online Library}
}

@article{gu2021impact,
  title={The impact of information technology usage on supply chain resilience and performance: An ambidexterous view},
  author={Gu, Minhao and Yang, Lu and Huo, Baofeng},
  journal={International Journal of Production Economics},
  volume={232},
  pages={107956},
  year={2021},
  publisher={Elsevier}
}

@article{li2017enhancement,
  title={Enhancement of supply chain resilience through inter-echelon information sharing},
  author={Li, Haobin and Pedrielli, Giulia and Lee, Loo Hay and Chew, Ek Peng},
  journal={Flexible Services and Manufacturing Journal},
  volume={29},
  pages={260--285},
  year={2017},
  publisher={Springer}
}

@ARTICLE{Ivanov_DSCM2021,
  author={Ivanov, Dmitry},
  journal={IEEE Transactions on Engineering Management}, 
  title={Digital Supply Chain Management and Technology to Enhance Resilience by Building and Using End-to-End Visibility During the COVID-19 Pandemic}, 
  year={2021},
  volume={},
  number={},
  pages={1-11},
  doi={10.1109/TEM.2021.3095193}}

@article{wong2021conceptual,
  title={A conceptual framework for information-leakage-resilience},
  author={Wong, Wai-Peng and Tan, Kim Hua and Govindan, Kannan and Li, Di and Kumar, Ajay},
  journal={Annals of Operations Research},
  pages={1--21},
  year={2021},
  publisher={Springer}
}

@article{gao2019disruption,
  title={Disruption risk mitigation in supply chains: The risk exposure index revisited},
  author={Gao, Sarah Yini and Simchi-Levi, David and Teo, Chung-Piaw and Yan, Zhenzhen},
  journal={Operations Research},
  volume={67},
  number={3},
  pages={831--852},
  year={2019},
  publisher={INFORMS}
}

@article{valuedisruptionduration,
author = {Mili Mehrotra and William Schmidt},
title ={The Value of Supply Chain Disruption Duration Information},

journal = {Production and Operations Management},
volume = {30},
number = {9},
pages = {3015-3035},
year = {2021},
doi = {10.1111/poms.13415},

URL = { 
    
        https://doi.org/10.1111/poms.13415
    
    

},
eprint = { 
    
        https://doi.org/10.1111/poms.13415
    
    

}
}

@article{SEREL2008115,
title = {Inventory and pricing decisions in a single-period problem involving risky supply},
journal = {International Journal of Production Economics},
volume = {116},
number = {1},
pages = {115-128},
year = {2008},
issn = {0925-5273},
doi = {https://doi.org/10.1016/j.ijpe.2008.07.012},
url = {https://www.sciencedirect.com/science/article/pii/S0925527308002363},
author = {Doğan A. Serel}
}

@article{Tomlin,
author = {Tomlin, Brian},
year = {2006},
month = {05},
pages = {639-657},
title = {On the Value of Mitigation and Contingency Strategies for Managing Supply Chain Disruption Risks},
volume = {52},
journal = {Management Science},
doi = {10.1287/mnsc.1060.0515}
}

@article{SixthTone2020,
  title = {From Shortage to Glut: China’s Mask Market Collapse},
  author = {{Sixth Tone}},
  year = {2020},
  url = {https://www.sixthtone.com/news/1005781},
  note = {Accessed: 2026-04-08}
}

@article{GUAN2020107450,
title = {Demand information sharing in competing supply chains with manufacturer-provided service},
journal = {International Journal of Production Economics},
volume = {220},
pages = {107450},
year = {2020},
issn = {0925-5273},
doi = {https://doi.org/10.1016/j.ijpe.2019.07.023},
url = {https://www.sciencedirect.com/science/article/pii/S0925527319302609},
author = {Zili Guan and Xumei Zhang and Maosen Zhou and Yiran Dan}}

@article{sahin2002flow,
  title={Flow coordination and information sharing in supply chains: review, implications, and directions for future research},
  author={Sahin, Funda and Robinson, E Powell},
  journal={Decision Sciences},
  volume={33},
  number={4},
  pages={505--536},
  year={2002},
  publisher={Wiley Online Library}
}

@article{zheng2019reactive,
  title={Reactive production with preprocessing restriction in supply chains with forecast updates},
  author={Zheng, Meimei and Shunko, Masha and Gavirneni, Nagesh and Shu, Yan and Wu, Kan},
  journal={IISE Transactions},
  volume={51},
  number={12},
  pages={1402--1436},
  year={2019},
  publisher={Taylor \& Francis}
}

@article{topan2018,
  title={Using imperfect advance demand information in lost-sales inventory systems with the option of returning inventory},
  author={Topan, Engin and Tan, Tarkan and van Houtum, Geert-Jan and Dekker, Rommert},
  journal={IISE Transactions},
  volume={50},
  number={3},
  pages={246--264},
  year={2018},
  publisher={Taylor \& Francis}
}

@article{Chen04052025,
author = {Zhen-Yu Chen and Minghe Sun},
title = {Data-driven contract design for supply chain coordination with algorithm sharing and algorithm competition},
journal = {IISE Transactions},
volume = {57},
number = {5},
pages = {573--589},
year = {2025},
publisher = {Taylor \& Francis},
doi = {10.1080/24725854.2024.2361460}}

@article{Wu01022022platform,
author = {Yuguang Wu and Qiao Chu He and Xin Wang},
title = {Competitive spatial pricing for urban parking systems: Network structures and asymmetric information},
journal = {IISE Transactions},
volume = {54},
number = {2},
pages = {186--197},
year = {2022},
publisher = {Taylor \& Francis},
doi = {10.1080/24725854.2021.1937755}}

@article{liu2026mechanism,
  title={Mechanism Design for Privacy-Preserving Information Sharing in Oligopoly Competition},
  author={Liu, Yuxin and Rahimian, M Amin},
  journal={arXiv preprint arXiv:2606.02348},
  year={2026}
}

\newpage
\appendix

\section{Prior and Posterior Beliefs}\label{Sec:Appx:probabilities}

This appendix presents probability expressions capturing firms' beliefs. Throughout this section, firms $i$ and $k$ denote platform members, while firms $u$ and $v$ denote non-members.

\subsection{\emph{Prior Probabilities}}

Parameter $\theta$ denotes the market state defined in Section~\ref{Sec:Supply_chain_setting}. A shock occurs with probability $\omega$, and conditional on a shock, the shock location is Market $j$ with probability $r_j$:
\begin{align}
\Pr(\theta=0) &= 1-\omega, &
\Pr(\theta=j) &= \omega r_j .\nonumber
\end{align}

Conditional on a shock at Market $j$, a non-member firm $u$ observes a signal with probability $p_u$:
\begin{align}
\Pr(z_u=(+,j)\mid \theta=j) &= p_u, &
\Pr(z_u=-\mid \theta=j) &= 1-p_u, &
\Pr(z_u=-\mid \theta=0)=1.\nonumber
\end{align}
For platform members, information sharing improves information availability
\begin{equation}
p_{\text{platform}}= 1-\prod_{i\in \mathbf{F}:V_i=1}(1-p_i). \nonumber
\end{equation}
Accordingly, for a platform member $i$,
\begin{align}
\Pr(z_i=(+,j)\mid \theta=j) &= p_{\text{platform}}, &
\Pr(z_i=-\mid \theta=j) &= 1-p_{\text{platform}}, \nonumber\\
\Pr(z_i=-\mid \theta=0) &= 1.\nonumber
\end{align}

Unconditional on a shock,
\begin{align}
\Pr(z_k=(+,j)) &= \omega r_j p_k, &
\Pr(z_i=(+,j)) &= \omega r_j p_{\text{platform}}. \nonumber
\end{align}

\subsection{\emph{Posterior beliefs over the market state}}
If a non-member firm $u$ does not observe a signal, its posterior belief over the market state is
\begin{align}
\Pr(\theta=j\mid z_u=-)
&=
\frac{(1-p_u)\omega r_j}{1-\omega p_u}, \nonumber\\
\Pr(\theta=0\mid z_u=-)
&=
\frac{1-\omega}{1-\omega p_u}.\nonumber
\end{align}

If a platform member $i$ does not observe a signal, then
\begin{align}
\Pr(\theta=j\mid z_i=-)
&=
\frac{(1-p_{\text{platform}})\omega r_j}{1-\omega p_{\text{platform}}}, \nonumber\\
\Pr(\theta=0\mid z_i=-)
&=
\frac{1-\omega}{1-\omega p_{\text{platform}}}.\nonumber
\end{align}

If a firm observes a shock signal at Market $j$, then it knows the market state with certainty:
\begin{equation}
\Pr(\theta=j\mid z_x=(+,j))=1,
\qquad x\in\{i,u\}.\nonumber
\end{equation}

\subsection{\emph{Posterior Beliefs over Other Firms' Types}}

Conditional on a shock at Market $j$, non-members’ signal realizations are independent.

\paragraph{Member beliefs.}
If platform member $i$ observes a signal at Market $j$, then for other member $k$, it has the following belief:
\begin{equation}
\Pr(z_k=(+,j)\mid z_i=(+,j))=1,\nonumber
\end{equation}
indicating information sharing within the platform. For a non-member $v$, the belief is
\begin{align}
\Pr(z_v=(+,j)\mid z_i=(+,j)) &= p_v, & 
\Pr(z_v=(+,j)\mid z_i=-)
&=
\frac{\omega r_j(1-p_{\text{platform}})p_v}{1-\omega p_{\text{platform}}}.\nonumber
\end{align}

\paragraph{Non-member beliefs.}
If non-member Firm $u$ observes a signal at Market $j$, then it knows $\theta=j$. In this case, for another non-member $v$ and a platform member $k$,
\begin{align}
\Pr(z_v=(+,j)\mid z_u=(+,j)) &= p_v,\nonumber\\
\Pr(z_k=(+,j)\mid z_u=(+,j)) &= p_{\text{platform}}.\nonumber
\end{align}

If Firm $u$ does not observe a signal, then
\begin{align}
\Pr(z_v=(+,j)\mid z_u=-)
&=
\frac{\omega r_j(1-p_u)p_v}{1-\omega p_u}, \\
\Pr(z_k=(+,j)\mid z_u=-)
&=
\frac{\omega r_j(1-p_u)p_{\text{platform}}}{1-\omega p_u}.
\end{align}

\section{Small Supply Chain: Firms' Equilibrium Quantity Decisions Analysis}

\label{App:Equilibrium}

In this section, we derive the firms' equilibrium production quantities in the two-firm supply chain setting. 

\subsection{\emph{Baseline: Equilibrium under Private Information}} 
\label{App:Equilibrium_base}

In the baseline case, firms operate without the platform, represented by 
$\vec{V} = [0,0]$. Therefore, demand signals remain private, and firms form beliefs about their competitor's information. We distinguish production decisions for each firm based on whether they observe the signal or not, and represent the quantity-game equilibrium as the vector
$$[q^*_1(-|[0,0]),q^*_1(+,A|[0,0]);q^*_2(-|[0,0]),q^*_2(+,A|[0,0])].$$ 

For compactness in this section, let 
$$
q^+_1 := q^*_{1A}(+,A\mid[0,0]), 
\quad 
q^+_2 := q^*_{2A}(+,A\mid[0,0]), 
\quad 
q^-_1:= q^*_{1A}(-\mid[0,0]), 
\quad 
q^-_2:= q^*_{2A}(-\mid[0,0]).
$$

The production quantities are nonnegative.

When Firm $i$ observes a signal, it chooses its production in response to expectation its competitor's production.
When Firm $i$ does not observe a signal, uncertainty remains both for market conditions and the competitor's information. The corresponding posterior beliefs are derived in Section \ref{Sec:Appx:probabilities}.

A Bayesian equilibrium is a vector $
(q^-_1,q^+_1,q^-_2,q^+_2)$
such that each quantity is optimal for the corresponding firm and information state, given the competitor's type-contingent production quantities. Hence, none of the firms has an incentive to deviate at the equilibrium.

\paragraph{Payoff Functions and Best Responses.}

We derive the equilibrium quantities in the baseline case.

When Firm $i$ observes the disruption, it chooses $q^+_{i}$ taking as given the quantities chosen by Firm $j$. Its expected payoff is
\begin{equation}    
{\small
\begin{aligned}
\mathbb{E}\!\left[\pi_i\!\left(q^+_{i},q^+_{j},q^-_{i},q^-_{j},\theta\right)
\mid z_i=(+,A),\vec{V}=[0,0]\right]
&=
(1-p_j)\left[a_A + D - b_A\big(q^+_{i} + q^-_{j}\big)\right] q^+_{i}
\\
&\quad+
p_j\left[a_A + D - b_A\big(q^+_{i} + q^+_{j}\big)\right] q^+_{i}
\\
&\quad - c_i \left(q^+_{i}\right)^2.
\end{aligned}
}
\end{equation}
\medskip

When Firm $i$ does not observe the disruption, it chooses $q^-_{i}$ based on its posterior beliefs. Its expected payoff is
\begin{equation}
{\small
\begin{aligned}
\mathbb{E}\!\left[\pi_i\!\left(q^+_{i},q^+_{j},q^-_{i},q^-_{j},\theta\right)
\mid z_i=-,\vec{V}=[0,0]\right]
&=
\frac{1-\omega}{1-\omega + \omega(1-p_i)}
\left[a_A - b_A\big(q^-_{i} + q^-_{j}\big)\right] q^-_{i}\\
&\quad+
\frac{\omega(1-p_i)(1-p_j)}
{1-\omega + \omega(1-p_i)}
\left[a_A + D - b_A\big(q^-_{i} + q^-_{j}\big)\right] q^-_{i}\\
&\quad+
\frac{\omega(1-p_i)p_j}
{1-\omega + \omega(1-p_i)}
\left[a_A + D - b_A\big(q^-_{i} + q^+_{j}\big)\right] q^-_{i}\\
&\quad - c_i \left(q^-_{i}\right)^2.
\end{aligned}
}
\end{equation}

\paragraph{First-Order and Second-Order Optimality Conditions}

Using the payoff functions above, Firm $i$'s first-order and second-order conditions are as follows.

\medskip

\noindent
\textit{When Firm $i$ observes the disruption:}

\begin{equation}\label{Eq:FOC1}
\begin{aligned}
\frac{\partial 
\mathbb{E}\!\left[\pi_i\!\left(q^+_{i},q^+_{j},q^-_{i},q^-_{j},\theta\right)
\mid z_i=(+,A),\vec{V}=[0,0]\right]}
{\partial q^+_{i}}
&=
a_A + D
- b_A\!\left(
p_j q^+_{j}
+ (1-p_j) q^-_{j}
\right) \\
&\quad - 2(b_A+c_i) q^+_{i},
\end{aligned}
\end{equation}
\begin{equation}
\begin{aligned}
\frac{\partial^2 
\mathbb{E}\!\left[\pi_i\!\left(q^+_{i},q^+_{j},q^-_{i},q^-_{j},\theta\right)
\mid z_i=(+,A),\vec{V}=[0,0]\right]}
{\partial \left(q^+_{i}\right)^2}
&=
-2(b_A+c_i)<0.
\end{aligned}
\end{equation}

\medskip

\noindent
\textit{When Firm $i$ does not observe the disruption:}

\begin{equation} \label{Eq:FOC2}
\begin{aligned}
\frac{\partial 
\mathbb{E}\!\left[\pi_i\!\left(q^+_{i},q^+_{j},q^-_{i},q^-_{j},\theta\right)
\mid z_i=-,\vec{V}=[0,0]\right]}
{\partial q^-_{i}}
&=
a_A
+\frac{\omega(1-p_i)}{1-\omega p_i}D
\\
&\quad
-\frac{1-\omega p_i-\omega p_j+\omega p_ip_j}
{1-\omega p_i}\,
b_A q^-_{j}
\\
&\quad
-\frac{\omega(1-p_i)p_j}{1-\omega p_i}\,
b_A q^+_{j}
\\
&\quad
-2(b_A+c_i)q^-_{i}.
\end{aligned}
\end{equation}
\begin{equation}    
\begin{aligned}
\frac{\partial^2 
\mathbb{E}\!\left[\pi_i\!\left(q^+_{i},q^+_{j},q^-_{i},q^-_{j},\theta\right)
\mid z_i=-,\vec{V}=[0,0]\right]}
{\partial \left(q^-_{i}\right)^2}
&=
-2(b_A+c_i)<0.
\end{aligned}
\end{equation}

\paragraph{Best Responses.} Since $-2(b_A+c_i)<0$, the first-order condition together with the nonnegativity constraint characterizes Firm $i$'s unique best responses for each type.

\begin{equation}
BR^+_{i}
=
\max\left\{
0,\;
\frac{
a_A+D-b_A\left(p_j q^+_j+(1-p_j)q^-_j\right)
}{
2(b_A+c_i)
}
\right\}.
\label{Eq:Q_withsignal_withoutplatform}
\end{equation}

\begin{equation}
BR^-_i
=
\max\left\{
0,\;
\frac{
a_A
+\frac{\omega(1-p_i)}{1-\omega p_i}D
-\frac{1-\omega p_i-\omega p_j+\omega p_ip_j}{1-\omega p_i}
b_Aq^-_j
-\frac{\omega(1-p_i)p_j}{1-\omega p_i}
b_Aq^+_j
}{
2(b_A+c_i)
}
\right\}.
\label{Eq:Q_withoutsignal_withoutplatform}
\end{equation}

The corresponding conditions for Firm $j$ are obtained analogously by interchanging indices.

\paragraph{Properties of Equilibrium Quantities.}

Every production quantity is constrained to be nonnegative. Depending on which nonnegativity constraints are binding, the system of best responses has $2^4$ possible active-set configurations. Instead of solving the system separately for all configurations, we first describe equilibrium properties to restrict the set of candidate equilibrium configurations.

For the equilibrium analysis, we introduce the following notation, which is useful in the proofs. For each firm $k\in\{i,j\}$, let
$$
\mu_k
=
\frac{\omega(1-p_k)}{1-\omega p_k}
$$
denote the posterior probability that the demand shock has occurred conditional on Firm $k$ not observing a signal. Let,

$$
\hat{a}_k^-:=a_A+\mu_kD,
\qquad
\hat{a}^+:=a_A+D,
$$
where $\hat{a}_k^-$ is the expected demand intercept conditional on Firm $k$ not observing the signal, whereas $\hat{a}^+$ is the demand intercept when the positive demand shock is observed. 
Finally, we define
$$\kappa_k:=2(b_A+c_k),$$
which is the coefficient on Firm $k$'s own quantity in its first-order condition.

From the parameter domains of our model,
\begin{equation} \label{eq:useful_coditions}
0<\mu_k<1,\qquad
\hat{a}_k^-<\hat{a}^+,\qquad
\kappa_k>2b_A,
\end{equation}

We first examine the quantities when firms observe the signal. %

\begin{lemma}
\label{lem:positive_signal_quantities}

 Firms with a disruption signal produce a positive quantity, i.e., $q_i^+,q_j^+>0$. 
\end{lemma}

\begin{proof}
Because all equilibrium quantities are nonnegative and the coefficients of the competitors' production quantities are negative in the best response, the best-response expressions \eqref{Eq:Q_withsignal_withoutplatform},\eqref{Eq:Q_withoutsignal_withoutplatform}, imply upper bounds
\begin{equation} \label{eq:upperbound}
q_k^-\leq \frac{\hat{a}_k^-}{\kappa_k},
\qquad
q_k^+\leq \frac{\hat{a}^+}{\kappa_k}.
\end{equation}

Consider Firm $i$'s marginal payoff at the boundary $q^+_i=0$, denoted by $MP^+_i$, which is

$$MP_i^+
:=
\left.
\frac{\partial
\mathbb{E}\!\left[\pi_i
\mid z_i=(+,A),\vec{V}=[0,0]\right]}
{\partial q_i^+}
\right|_{q_i^+=0}=
\hat{a}^+
-
b_A\left[(1-p_j)q_j^-+p_jq_j^+\right].$$

Applying the upper bounds in \eqref{eq:upperbound} to Firm $j$ gives
$$MP_i^+
\geq
\hat{a}^+
-
\frac{b_A}{\kappa_j}
\left(
(1-p_j)\hat{a}_j^-+p_j\hat{a}^+
\right).$$

By conditions \eqref{eq:useful_coditions},
$$
MP_i^+
>
\hat{a}^+
-
\frac{b_A}{\kappa_j}\hat{a}^+.
$$

Moreover $\dfrac{b_A}{\kappa_j}<\dfrac12$ implies
$$
MP_i^+
>
\frac{\hat{a}^+}{2}
>0.$$

Thus, Firm $i$'s conditional payoff is strictly increasing in $q_i^+$ at the boundary $q_i^+=0$. Since its conditional payoff is strictly concave
in $q_i^+$, $q_i^+=0$ cannot be optimal. Hence
$$
q_i^+>0.
$$

The same argument applies to Firm $j$, by interchanging $i$ with $j$,
$$
q_j^+>0.
$$

\end{proof}

By Lemma \ref{lem:positive_signal_quantities}, any equilibrium quantity vector with a zero-production component can have that zero only in one of the without-signal quantities, $q_i^-$ and $q_j^-$.

We continue the analysis by considering two cases according to the firms' information power. First, we consider firms with symmetric information power. Then, we consider firms with asymmetric information power. Without loss of generality, in the asymmetric case we assume $p_i>p_j$.

\medskip

\noindent
\textit{When Firms have symmetric information power ($p_i=p_j=p$):}

\begin{lemma}
\label{lem:symmetric_information}

When firms have the same information power ($p_1=p_2$), their without-signal production quantities are strictly positive, i.e., $q_i^->0$ and $q_j^->0$.

\end{lemma}

\begin{proof}

Since $p_i=p_j=p$, we have
$$
\mu_i=\mu_j=:\mu, \text{ and therefore } \hat a_i^-=\hat a_j^-=: \hat a^-.
$$

Consider Firm $i$'s marginal payoff at the boundary $q^-_i=0$, denoted by $MP^-_i$, which is

$$ MP_i^-:=\left.\frac{\partial \mathbb{E}\!\left[\pi_i\mid z_i=-,\vec{V}=[0,0]\right]} {\partial q_i^-} \right|_{q_i^-=0}= \hat a^- - b_A\left[(1-rp)q_j^-+rpq_j^+\right].$$

Applying the upper bounds in \eqref{eq:upperbound},
$$
MP_i^- \geq \hat a^- - \frac{b_A}{\kappa_j} \left[(1-rp)\hat a^-+rp\hat a^+\right].
$$

Since $\hat a^+-\hat a^-=(1-r)D$, we have
$$
(1-\mu p)\hat a^-+\mu p\hat a^+
=
\hat a^-+\mu p(1-\mu)D.
$$

Moreover, $p(1-\mu)<1$, so 
$$
\mu p(1-\mu)D<\mu D<\hat a^-.
$$ 

Hence, 
$$
(1-\mu p)\hat a^-+\mu p\hat a^+<2\hat a^-.
$$

Since $\dfrac{b_A}{\kappa_j}<\dfrac{1}{2}$, we have
$$
MP_i^->0.
$$

Thus, Firm $i$'s payoff is strictly increasing in $q_i^-$ at the boundary $q_i^-=0$. Since its payoff conditioned on not observing a signal is strictly concave in $q_i^-$, $q_i^-=0$ cannot be optimal. Hence,
$$
q_i^->0.
$$

By symmetry,
$$
q_j^->0.
$$

\end{proof}

\noindent
\textit{When Firms have asymmetric information power ($p_i>p_j$):}

\begin{lemma}
\label{lem:lower_information_positive}

When firms have asymmetric information power, $p_i>p_j$, the lower-information firm's without-signal production quantity is positive, i.e., $ q_j^->0$.

\end{lemma}

\begin{proof}

Recall that
$$
\mu(p)=\frac{\omega(1-p)}{1-\omega p}
$$
is strictly decreasing in $p$. Thus,
$$
p_i>p_j
\quad\Longrightarrow\quad
\mu_i<\mu_j.
$$

Consider Firm $j$'s marginal payoff at the boundary $q_j^-=0$:
$$
MP_j^- := \left. \frac{\partial \mathbb{E}\!\left[\pi_j \mid z_j=-,\vec{V}=[0,0]\right]} {\partial q_j^-} \right|_{q_j^-=0} = \hat a_j^- - b_A\left[ (1-\mu_jp_i)q_i^-+\mu_jp_iq_i^+ \right].
$$

Applying the upper bounds in \eqref{eq:upperbound},
$$
MP_j^- \geq \hat a_j^- - \frac{b_A}{\kappa_i} \left[ (1-\mu_jp_i)\hat a_i^-+\mu_jp_i\hat a^+ \right].
$$

Since $\mu_i<\mu_j$ and
$p_i(1-\mu_i)<1$,
\begin{equation}\notag
\begin{aligned}
(1-\mu_jp_i)\hat a_i^-+\mu_jp_i\hat a^+ &= a_A+\left[\mu_i+\mu_jp_i(1-\mu_i)\right]D \\
&< a_A+2\mu_jD \\
&< 2\hat a_j^-,
\end{aligned}
\end{equation}

Since $\dfrac{b_A}{\kappa_i}<\dfrac12$, we have

$$MP_j^->0.$$

Thus, Firm $j$'s payoff is strictly increasing in
$q_j^-$ at the boundary $q_j^-=0$. Since its conditional payoff is strictly concave in $q_j^-$, $q_j^-=0$ cannot be optimal. Hence,
$$
q_j^->0.
$$\end{proof}

\paragraph{Equilibrium Candidates.}

By Lemmas \ref{lem:positive_signal_quantities}-\ref{lem:lower_information_positive}, when firms have the same information power, the equilibrium production quantities under each signal type are strictly positive for each firm. In contrast, when firms have asymmetric information power, the higher-information firm's without-signal quantity may be zero. Thus, the equilibrium can be either an \textit{interior} solution or a \textit{boundary} solution in which only the higher-information firm's without-signal quantity is zero.

Next, we solve the candidate solutions and identify the condition under which each candidate is an equilibrium.

\paragraph{Interior Candidate Equilibrium}

We first solve the system for the interior candidate, in which all four production
quantities are strictly positive. Hence, the nonnegativity constraints in
the best responses are not binding, and the interior candidate can be obtained by
solving the unconstrained system of first-order conditions.

Subtracting each firm's the without-signal first order condition from its the with-signal first order condition gives
\begin{align}
\kappa_i(q_i^+-q_i^-) +b_Ap_j(1-\mu_i)(q_j^+-q_j^-) &= \hat a^+-\hat a_i^-, \label{eq:interior-diff-i} \\
b_Ap_i(1-\mu_j)(q_i^+-q_i^-) +\kappa_j(q_j^+-q_j^-) &= \hat a^+-\hat a_j^-.
\label{eq:interior-diff-j}
\end{align}

Next, define the changes in production after observing the signal as
$$
x_i:=q_i^+-q_i^-,
\qquad
x_j:=q_j^+-q_j^-.
$$

Then, \eqref{eq:interior-diff-i} and \eqref{eq:interior-diff-j} can be written as 
\begin{equation}
\begin{pmatrix}
\kappa_i & b_Ap_j(1-\mu_i) \\
b_Ap_i(1-\mu_j) & \kappa_j
\end{pmatrix}
\begin{pmatrix}
x_i\\
x_j
\end{pmatrix}
=
\begin{pmatrix}
(1-\mu_i)D\\
(1-\mu_j)D
\end{pmatrix}.
\label{eq:signal-change-system}
\end{equation}

Since $
\kappa_i\kappa_j>4b_A^2$ and $ 0<p_ip_j(1-\mu_i)(1-\mu_j)<1$,
the coefficient matrix is nonsingular, and the system has a unique solution. Hence,
\begin{align}
x_i
&=
\frac{
(1-r_i)D
\left(\kappa_j-b_Ap_j(1-r_j)\right)
}{
\kappa_i\kappa_j
-b_A^2p_ip_j(1-r_i)(1-r_j)
},
\label{eq:xi-interior}
\\
x_j
&=
\frac{
(1-r_j)D
\left(\kappa_i-b_Ap_i(1-r_i)\right)
}{
\kappa_i\kappa_j
-b_A^2p_ip_j(1-r_i)(1-r_j)
}.
\label{eq:xj-interior}
\end{align}

Next, substituting $q_i^+=q_i^-+x_i$, and $q_j^+=q_j^-+x_j$, into the without-signal first-order conditions gives
\begin{equation}
\begin{pmatrix}
\kappa_i & b_A\\
b_A & \kappa_j
\end{pmatrix}
\begin{pmatrix}
q_i^-\\
q_j^-
\end{pmatrix}
=
\begin{pmatrix}
\hat a_i^- - b_A \mu_i p_j x_j\\
\hat a_j^- - b_A \mu_j p_i x_i
\end{pmatrix}.
\label{eq:no-signal-interior-system}
\end{equation}

Since $\kappa_i\kappa_j>4b_A^2$, the coefficient matrix is nonsingular, and the system has a unique solution. Hence,
\begin{align}
q_i^-
&= \frac{ \kappa_j\left(\hat a_i^- - b_A \mu_i p_j x_j\right) - b_A\left(\hat a_j^- - b_A \mu_j p_i x_i\right) }{ \kappa_i\kappa_j-b_A^2 },
\label{eq:qi-minus-interior}
\\
q_j^- &= \frac{ \kappa_i\left(\hat a_j^- - b_A \mu_j p_i x_i\right) - b_A\left(\hat a_i^- - b_A \mu_i p_j x_j\right) }{ \kappa_i\kappa_j-b_A^2}.
\label{eq:qj-minus-interior}
\end{align}

Finally, the with-signal quantities are obtained from
\begin{equation}
q_i^+=q_i^-+x_i,
\qquad
q_j^+=q_j^-+x_j.
\label{eq:plus-interior}
\end{equation}

Thus, equations \eqref{eq:xi-interior}- \eqref{eq:xj-interior} and \eqref{eq:qi-minus-interior}- \eqref{eq:plus-interior} define the unique interior candidate.

\paragraph{Boundary Candidate Equilibrium}

Next, we solve the system for the boundary candidate by fixing $q_i^-=0$.

The remaining three first-order conditions are

\begin{align}
\kappa_j q_j^-+b_A \mu_jp_i q_i^+ &= \hat a^-_j, \label{eq:boundary_j_minus} \\
\kappa_i q_i^+ + b_A\left( (1-p_j)q_j^-+p_jq_j^+ \right) &= \hat a^+, \label{eq:boundary_i_plus} \\
\kappa_j q_j^++b_Ap_iq_i^+ &= \hat a^+.
\label{eq:boundary_j_plus}
\end{align}

From \eqref{eq:boundary_j_minus} and \eqref{eq:boundary_j_plus},
\begin{equation}
q_j^-= \frac{\hat a_j^- - b_A \mu_jp_iq_i^+}{\kappa_j}, \qquad q_j^+= \frac{\hat a^+ - b_Ap_iq_i^+}{\kappa_j}.
\label{eq:boundary-j-solutions}
\end{equation}

Substituting these expressions into
\eqref{eq:boundary_i_plus} gives
$$
\left( \kappa_i\kappa_j - b_A^2p_i \left( p_j+\mu_j(1-p_j) \right) \right)q_i^+ = a_A(\kappa_j-b_A) + D\left( \kappa_j b_A\left( p_j+\mu_j(1-p_j) \right) \right).
$$

Since
$$
\kappa_i\kappa_j -b_A^2p_i\left(p_j+\mu_j(1-p_j)\right) >4b_A^2-b_A^2>0,
$$
the system has a unique solution. In particular,
\begin{equation}
q_i^+ = \frac{ (\kappa_j-b_Ap_j)\hat a^+ -b_A(1-p_j)\hat a_j^- }{ \kappa_i\kappa_j -b_A^2p_i\bigl(p_j+\mu_j(1-p_j)\bigr) }.
\label{eq:boundary-i-plus-solution}
\end{equation}

By \eqref{eq:boundary-j-solutions}, the remaining quantities are

\begin{align}
q_j^+ =& \frac{ \kappa_j(\kappa_i-b_Ap_i)\hat a^+ +b_A^2p_i(1-p_j)a_A(1-\mu_j) }{ \kappa_j\left( \kappa_i\kappa_j -b_A^2p_i\bigl(p_j+\mu_j(1-p_j)\bigr) \right) },\label{eq:boundary-j plus-solution}\\
q_j^- =& \frac{ a_A\left( \kappa_j(\kappa_i-b_A \mu_jp_i) -b_A^2p_ip_j(1-\mu_j) \right) +\kappa_j \mu_jD(\kappa_i-b_Ap_i) }{ \kappa_j\left( \kappa_i\kappa_j -b_A^2p_i\bigl(p_j+\mu_j(1-p_j)\bigr) \right) }. 
\label{eq:boundary-j-minus-solution}   
\end{align}

Since $\kappa_i,\kappa_j>2b_A$, $0<p_i,p_j,\mu_j<1$, and
$0<\hat a_j^-<\hat a^+$, the quantity expressions are strictly positive. So,
$$
q_i^+>0,
\qquad
q_j^->0,
\qquad
q_j^+>0.
$$

\paragraph{Equilibrium Characterization}

From the equilibrium candidate analysis above, equilibrium selection reduces to checking whether $q_i^-=0$ is optimal for Firm $i$ when $p_i>p_j$. Firm $i$'s without-signal marginal payoff at the boundary $q_i^-=0$ is

\begin{equation} 
MP_i^- = \hat a_i^- - b_A\left( (1-\mu_ip_j)q_j^-+\mu_ip_jq_j^+ \right). 
\label{eq:boundary-marginal} 
\end{equation}

Since Firm $i$'s payoff is strictly concave in $q_i^-$, the KKT condition at the boundary is 
\begin{equation}
    q^-_i=0 \iff  MP_i^-\leq 0
\end{equation}

By substituting the boundary candidate quantities, we have

$$\begin{aligned}
    MP_i^-=&\frac{
a_A(\kappa_j-b_A)
\left(
\kappa_i\kappa_j
-b_A^2p_ip_j(1-\mu_i)(1-\mu_j)
\right)
}{
\kappa_j
\left(
\kappa_i\kappa_j
-b_A^2p_i\bigl(p_j+\mu_j(1-p_j)\bigr)
\right)
}\\
&-\frac{
D\kappa_j
\left(
\kappa_i b_A \mu_j
+\kappa_i b_Ap_j\mu_i(1-\mu_j)
-\kappa_i\kappa_j\mu_i
-b_A^2p_i\mu_j(1-\mu_i)
\right)
}{
\kappa_j
\left(
\kappa_i\kappa_j
-b_A^2p_i\bigl(p_j+\mu_j(1-p_j)\bigr)
\right)
}
\end{aligned}$$

The numerator is linear in the market size $a_A$ and the shock size $D$. To simplify the equilibrium condition, denote their respective coefficients by
\begin{align}
\phi_a
&:=
(\kappa_j-b_A)
\left(
\kappa_i\kappa_j
-b_A^2p_ip_j(1-\mu_i)(1-\mu_j)
\right),
\label{eq:coefficient-a}
\\
\phi_D
&:=
\kappa_j
\left(
\kappa_i b_A \mu_j
+\kappa_i b_Ap_j\mu_i(1-\mu_j)
-\kappa_i\kappa_j\mu_i
-b_A^2p_i\mu_j(1-\mu_i)
\right).
\label{eq:coefficient-D}
\end{align}

Then
\begin{equation}
MP_i^-
=
\frac{
a_A\phi_a-D\phi_D
}{
\kappa_j
\left(
\kappa_i\kappa_j
-b_A^2p_i\bigl(p_j+\mu_j(1-p_j)\bigr)
\right)
}.
\label{eq:boundary-marginal-simplified}
\end{equation}

The denominator in \eqref{eq:boundary-marginal-simplified} is strictly
positive as shown earlier in the boundary candidate equilibrium analysis. Moreover, $\phi_a>0$ since $\kappa_j>b_A$ and $\kappa_i\kappa_j
-b_A^2p_ip_j(1-\mu_i)(1-\mu_j)>0$. 

From this analysis, we have the following 
complete equilibrium characterization for $\vec{V}=[0,0]$.

Let $\phi_a$ and $\phi_D$ be defined in
\eqref{eq:coefficient-a} and \eqref{eq:coefficient-D} and without loss of generality $p_i\geq p_j$. Then,

\begin{equation}
\mathbf q^*
=
\begin{cases}
\mathbf q^C,
& p_i>p_j,\ \phi_D>0,\ 
\dfrac{D}{a_A}\geq\dfrac{\phi_a}{\phi_D},\\
\mathbf q^I,
& \text{otherwise.}
\end{cases}
\label{eq:complete-equilibrium}
\end{equation}

\subsection{\emph{Equilibrium Under Information Sharing}}

When both firms join the platform ($V=[1,1]$), they share demand signals, eliminating the information asymmetry present in the baseline case. As a result, firms have identical information sets, and neither firm has any uncertainty about their competitor's information. The equilibrium quantities are represented by
$$
(q^*_1(-|[1,1]),q^*_1(+,A|[1,1]),q^*_2(-|[1,1]),q^*_2(+,A|[1,1])).
$$
For compactness, let
$$
q_i^{+p}:=q^*_{iA}(+,A\mid[1,1]),
\qquad
q_i^{-p}:=q^*_{iA}(-\mid[1,1]),
\qquad i\in\{1,2\}
$$

Conditional on the shared information, depending on whether a signal is observed, only one pair of quantities is realized. Consequently, the problem can be decomposed into two Cournot games: one conditional on the observed signal and one conditional on the absence of the signal. In both cases, firms have the same information about the demand environment and therefore face the same conditional demand intercept. The firms still choose quantities strategically and simultaneously, taking the other firm's quantity as given when computing their best response.

Firm $i$'s payoff is
$$
\pi_i(q_i,q_j)
=
\left[
\hat a-b_A(q_i+q_j)
\right]q_i-c_iq_i^2,
$$
where $\hat a$ denotes the relevant conditional demand intercept. Therefore, the unique pure-strategy Cournot equilibrium with linear inverse demand and quadratic costs is given by
\begin{equation}
q_{iA}^{*}
=
\frac{(b_A+2c_j)\hat a}
{4(b_A+c_i)(b_A+c_j)-b_A^2},
\qquad
q_{jA}^{*}
=
\frac{(b_A+2c_i)\hat a}
{4(b_A+c_i)(b_A+c_j)-b_A^2}.
\end{equation}

Hence, when at least one of the firms observes the shock and the signal is shared, the demand state is common knowledge and $\hat a=a_A+D$. Therefore,
\begin{equation}
q_{i}^{+p}
=
\frac{(b_A+2c_j)(a_A+D)}
{4(b_A+c_i)(b_A+c_j)-b_A^2}
\qquad \text{and} \qquad
q_{j}^{+p}
=
\frac{(b_A+2c_i)(a_A+D)}
{4(b_A+c_i)(b_A+c_j)-b_A^2}.
\end{equation}

When neither of the firms observes a signal, the uncertainty regarding the shock remains. However, because both firms have the same information, they share the common posterior belief about the market state. The posterior probability that the shock occurred is
$$
\Pr(\theta\neq 0\mid z_1=-,z_2=-)
=
\frac{\omega(1-p_1)(1-p_2)}
{1-\omega(p_1+p_2-p_1p_2)}.
$$
Thus, the corresponding conditional expected demand intercept is
$$
\hat a
=
a_A
+
\frac{\omega(1-p_1)(1-p_2)}
{1-\omega(p_1+p_2-p_1p_2)}D.
$$
The expected-profit game is therefore equivalent to a standard Cournot game with this conditional expected demand intercept. It follows that
\begin{equation}
q_{i}^{-p}
=
\frac{
(b_A+2c_j)
\left[
a_A+
\frac{\omega(1-p_1)(1-p_2)}
{1-\omega(p_1+p_2-p_1p_2)}D
\right]
}
{4(b_A+c_i)(b_A+c_j)-b_A^2}
\quad \text{and} \quad
q_{j}^{-p}
=
\frac{
(b_A+2c_i)
\left[
a_A+
\frac{\omega(1-p_1)(1-p_2)}
{1-\omega(p_1+p_2-p_1p_2)}D
\right]
}
{4(b_A+c_i)(b_A+c_j)-b_A^2}.
\end{equation}

\section{Participation Analysis In 2-Firm 1-Market Supply chain} \label{App:participation_analysis}

\subsection{\emph{Analytical Derivation of $\Delta_i$}} \label{App:participation_gain_derivation}

In this appendix, we derive an analytical representation of
$\Delta_i$ that will be used in the participation analysis. The equilibrium quantities under the two participation configurations ($\vec V=[0,0]$ and $\vec V=[1,1]$) are derived in Appendix~\ref{App:Equilibrium}. 
Since the quantity equilibrium without the platform may be either interior ($\mathbf{q}^I$) or boundary ($\mathbf{q}^C$), the participation gain is also piecewise. Let $\Delta_i^I$ and $\Delta_i^C$ denote the participation gain when $\mathbf q^I$ and $\mathbf q^C$ are the equilibrium quantity vectors under
$\vec V=[0,0]$, respectively.

Thus,
\begin{equation}
\Delta_i=
\begin{cases}
\Delta_i^I, & \text{if $\mathbf{q}^I$ is the equilibrium for $\vec V=[0,0]$},\\
\Delta_i^C, & \text{if $\mathbf{q}^C$ is the equilibrium for $\vec V=[0,0]$}.
\end{cases}
\end{equation}

After substituting the corresponding equilibrium quantities and simplifying, each case can be written as a rational function of the model parameters:
\begin{equation}
\Delta_i^I=\frac{f_i^I(\cdot)}{g_i^I(\cdot)},
\qquad
\Delta_i^C=\frac{f_i^C(\cdot)}{g_i^C(\cdot)}.
\end{equation}

We first show that the denominator of $\Delta_i$ is strictly positive.

\begin{lemma}\label{lem:denominator_sign}
    For any $b_A>0,\omega \in(0,1),c_i>0,p_i\in(0,1),i\in\{1,2\}$, the denominator polynomial is strictly positive:  $g^I(\cdot),g^C(\cdot)>0$.
\end{lemma}

\begin{proof}
We prove the result for the interior and boundary cases separately.

Starting with the interior case, write the denominator as the product of three terms:

$$
g^I(\cdot)=g_1(\cdot)g_2(\cdot)g_3(\cdot),
$$

where

$$
g_1(\cdot)=\left(3b_A^2+4c_1c_2+4b_A(c_1+c_2)\right)^2,
$$

$$
g_2(\cdot)=1-\omega(p_1+p_2-p_1p_2),
$$

and

$$
g_3(\cdot)
=
\Big(
4(b_A+c_1)(b_A+c_2)(1-p_1\omega)(1-p_2\omega)
-b_A^2p_1p_2(1-\omega)^2
\Big)^2.
$$

Since $b_A,c_1,c_2>0$, we have $g_1(\cdot)>0$. Moreover, because $\omega,p_1,p_2\in(0,1)$

$$
g_2(\cdot)
=1-\omega(p_1+p_2-p_1p_2)
=1-\omega\left(1-(1-p_1)(1-p_2)\right)>1-\omega>0.
$$

It remains to show that $g_3(\cdot)>0$. Since $c_1,c_2,b_A>0$ and $p_1,p_2,\omega\in(0,1)$,

$$
\begin{aligned}
&4(b_A+c_1)(b_A+c_2)(1-p_1\omega)(1-p_2\omega)
-b_A^2p_1p_2(1-\omega)^2\\
&\qquad>
b_A^2(1-\omega)^2(4-p_1p_2)>0.
\end{aligned}
$$

Hence, the expression inside the square is strictly positive, so $g_3(\cdot)>0$. Therefore,

\begin{equation}
  g^I(\cdot)>0.
\end{equation}

Similarly, in the boundary case, write the denominator as
$$
g^C(\cdot)=g_4(\cdot)g_5(\cdot)g_6(\cdot),
$$
where
$$
g_4(\cdot)=\left(3b_A^2+4b_Ac_1+4b_Ac_2+4c_1c_2\right)^2,
$$
$$
g_5(\cdot)=1-p_1\omega-p_2\omega+p_1p_2\omega,
$$
and
$$
g_6(\cdot)=
\Big(
4(c_1c_2+b_Ac_1+b_Ac_2+b_A^2)(-1+p_2\omega)
+b_A^2p_1(p_2+\omega-2p_2\omega)
\Big)^2.
$$

Note that $g_4(\cdot)=g_1(\cdot)>0$ and $g_5(\cdot)=g_2(\cdot)>0$. It remains to show that $g_6(\cdot)>0$.  Since $c_1,c_2,b_A>0$ and $p_1,p_2,\omega\in(0,1)$,

$$
\begin{aligned}
&4(c_1c_2+b_Ac_1+b_Ac_2+b_A^2)(-1+p_2\omega)
+b_A^2p_1(p_2+\omega-2p_2\omega)\\
&\qquad<
4b_A^2(-1+p_2\omega)
+b_A^2(p_2+\omega-2p_2\omega)\\
&\qquad=
-b_A^2(4-p_2-\omega-2p_2\omega)<0.
\end{aligned}
$$

Thus, the expression inside the square is strictly negative, and hence $g_6(\cdot)>0$. Therefore,

\begin{equation}
  g^C(\cdot)>0.
\end{equation}

\end{proof}

By Lemma \ref{lem:denominator_sign}, the denominator is strictly positive over the parameter domain. Therefore,
$$
\Delta_i^I\geq 0
\quad\Longleftrightarrow\quad
f_i^I(\cdot)\geq 0,
\qquad
\Delta_i^C\geq 0
\quad\Longleftrightarrow\quad
f_i^C(\cdot)\geq 0.
$$
Hence, the firm participation is determined by the sign of the corresponding numerator.

We first consider the interior case. After rearranging the numerator $f_i^I(\cdot)$ with respect to $a_A$ and $D$, it can be written as

\begin{equation}\label{eq:f_interior}
    f_i^I(\cdot)=\alpha_0(b_A,c_1,c_2,p_1,p_2,\omega) D^2,
\end{equation}

We next consider the boundary case. Rearranging numerator $f_i^C(\cdot)$ with respect to $a_A$ and $D$, gives a quadratic function
\begin{equation}
f_i^C(\cdot)
=
\alpha_1 a_A^2+\alpha_2 a_A D+\alpha_3 D^2,  
\end{equation}
The coefficients $\alpha_0,\alpha_1,\alpha_2$, and $\alpha_3$
depend only on $b_A,c_1,c_2,p_1,p_2$, and $\omega$.

The participation analysis in Section \ref{Sec:Analysis} and its proofs in Appendix \ref{App:participation_proofs} are derived from the numerator of the interior and boundary representations of the gain function.

\subsection{\emph{Proofs of Participation Analysis}} \label{App:participation_proofs}

\begin{proof}[Proof of Proposition \ref{prop:homogeneousregions}]

Under homogeneous firms,
$$
c_1=c_2=c,
\qquad
p_1=p_2=p.
$$
Since information power is symmetric, the private-information equilibrium
is interior. By Lemma~\ref{lem:denominator_sign}, the denominator of the
participation gain is strictly positive. Hence, participation is determined
by the sign of the numerator of the interior participation gain.

Under $c_1=c_2=c$ and $p_1=p_2=p$, factoring the numerator gives
\begin{equation}\label{Eq:homogeneous_numerator}
\begin{aligned}
&D^2(b_A+c)p(1-\omega)^2\omega (p-1)
\Big(
b_A^2-4b_Ac-4c^2-3b_A^2p-4b_Acp-b_A^2p^2
\\
&\qquad
-3b_A^2p\omega+4b_Acp\omega+4c^2p\omega
+6b_A^2p^2\omega+4b_Acp^2\omega
\Big).
\end{aligned}
\end{equation}

Since
$$
D^2(b_A+c)p(1-\omega)^2\omega(p-1)<0,
$$
the sign of the numerator is opposite to the sign of the term in the last
parentheses of expression \eqref{Eq:homogeneous_numerator}. Substituting $b_A=yc$ and factoring out $c^2>0$, firms
participate if and only if
\begin{equation}
\label{eq:homogeneous_participation_condition}
\begin{aligned}
&4+4y+4py-y^2+3py^2+p^2y^2\\
&\qquad
-p\omega
\left(
4+4y+4py-3y^2+6py^2
\right)
\geq0.
\end{aligned}
\end{equation}

The expression in
\eqref{eq:homogeneous_participation_condition} is affine in $\omega$. Let, $A(y,p)$ be the $\omega$ independent component and $B(y,p)$ be the $\omega$ dependent component.

Therefore, the participation condition can be written as
\begin{equation}
\label{eq:homogeneous_AB}
A(y,p)-\omega B(y,p)\geq0.
\end{equation}

By rearranging terms, we have 
\begin{align}
A(y,p) =&4+4y(1+p)+y^2(-1+3p+p^2), \label{eq:windep}\\
B(y,p)=& p(3y+2)\left(2+y(-1+2p)\right). \label{eq:wdepend}
\end{align}

Because \eqref{eq:homogeneous_AB} is affine in $\omega$, we first
determine when its values at both endpoints with respect to $\omega$ are positive for every
$p\in(0,1)$. This gives the parameter region for $y>0$ independent of $p$ and $\omega$.

At $\omega=0$, since
$$
\frac{\partial A(y,p)}{\partial p}=4y+3y^2+2py^2>0,
$$
$A(y,p)$ is strictly increasing in $p$. Hence, its sign for all
$p\in(0,1)$ is positive if 
$$A(y,0)=4+4y-y^2\geq0.$$

At $\omega=1$, the expression is
\begin{equation}
\label{eq:homogeneous_omega_one}
A(y,p)-B(y,p)
=
(1-p)
\left(
4+4y-y^2+p(4y+5y^2)
\right).
\end{equation}
Since $0<p<1$ and $4y+5y^2>0$, $A(y,p)-B(y,p)$ is also strictly
positive for every $p\in(0,1)$ if $4+4y-y^2\geq0$.

For $y>0$,
$$
4+4y-y^2\geq0
\quad\Longleftrightarrow\quad
y\leq2(1+\sqrt{2}).
$$
When $y\leq2(1+\sqrt{2})$, the participation expression \eqref{eq:homogeneous_participation_condition} is strictly positive at both $\omega=0$ and $\omega=1$. Since it is affine in $\omega$, it is
strictly positive for every $\omega\in(0,1)$. Hence firms participate
if $y\leq2(1+\sqrt{2})$, or equivalently,

\begin{equation}\label{eq:prop_homogeneous_item1}
    c\geq\frac{1}{2(1+\sqrt{2})}\,b_A.
\end{equation}
This proves that condition \text{(i)} is sufficient for platform formation.\\

\medskip

Next, we consider the case where condition (i) does not hold, $y>2(1+\sqrt{2})$. We have

$$
A(y,0)=4+4y-y^2<0.
\qquad \text{and}\qquad
A(y,1)=4+8y+3y^2>0.
$$

Since $A(y,p)$ is strictly increasing in $p$, there is a unique root, denoted by $\bar p(y)$, of $A(y,p)=0$ with respect to
$p\in(0,1)$.
Solving $A(y,p)=0$ for $p$ gives
\begin{equation}
    \bar p(y)=\frac{1}{2}\sqrt{\frac{8+13y}{y}}-\frac{4+3y}{2y}.
\end{equation}

Consequently, in the region $y>2(1+\sqrt{2})$,
\begin{equation}
\label{eq:A_pbar}
\begin{aligned}
    p> \bar{p}(y) \iff A(y,p)> 0,\\   p= \bar{p}(y) \iff A(y,p)= 0, \\ 0<p< \bar{p}(y) \iff A(y,p)< 0.
\end{aligned}
\end{equation}

We next show that $A(y,p)\geq0$ ensures participation when $\omega>0$. By \eqref{eq:wdepend}, since
$p>0$ and $3y+2>0$, the sign of $B(y,p)$ is determined by
\begin{equation}\label{Eq:before_phat}
2+y(-1+2p).
\end{equation}
Expression \eqref{Eq:before_phat} has a unique root, denoted by $\hat{p}(y)$ as
$$\hat{p}(y)=\frac{y-2}{2y}.$$
Since expression \eqref{Eq:before_phat} is linear and strictly increasing in $p$, this gives
\begin{equation} \label{Eq:phat_condition}
\begin{aligned}
    p> \hat{p}(y) \iff B(y,p)> 0,\\   p= \hat{p}(y) \iff B(y,p)= 0, \\ 0<p< \hat{p}(y) \iff B(y,p)< 0.
\end{aligned}
\end{equation}

At $p=\hat{p}(y)=\dfrac{y-2}{2y}$,
$$
A\left(y,\hat{p}(y)\right)
=
\frac{(y+2)(3y+2)}{4}>0.
$$
Since $A(y,p)$ is strictly increasing in $p$ and $A(y,\bar p(y))=0$, it implies that
\begin{equation}
\label{eq:pbar_order}
\bar p(y)<\hat p(y).
\end{equation}

Next consider $p\geq\bar p(y)$. If $\bar p(y) \leq p \leq\hat{p}(y)$, then, \eqref{eq:A_pbar} and \eqref{Eq:phat_condition} give $A(y,p)\geq0$ and $B(y,p)\leq0$.
Therefore, 
\begin{equation}
A(y,p)-\omega B(y,p)\geq 0 \qquad  \iff \qquad \bar p(y)\leq p \leq \hat p(y).
\end{equation}

Next, if $p>\hat{p}(y)$, then $B(y,p)>0$, so $A(y,p)-\omega B(y,p)$ is decreasing in $\omega$. So the smallest
endpoint value is attained at $\omega=1$. Recall expression
\eqref{eq:homogeneous_omega_one},
$$
A(y,p)-B(y,p) = (1-p)\left(4+4y-y^2+p(4y+5y^2)\right).
$$
At $p=\hat{p}(y)=\frac{y-2}{2y}$, it is equal to
$$
A(y,\hat{p}(y))-B(y,\hat{p}(y)) =(1-\hat{p}(y)) \big(
4+4y-y^2
+\frac{y-2}{2y}(4y+5y^2)\big)
=
(1-\hat{p}(y))\frac{y(3y+2)}{2}>0.
$$
The expression inside the second parenthesis of \eqref{eq:homogeneous_omega_one} is strictly increasing in $p$. Thus, when  
\begin{equation}
A(y,p)-B(y,p)>0 \qquad \text{if} \qquad p>\hat{p}.
\end{equation}
Hence, in both cases,
$$
    A(y,p)-\omega B(y,p)\geq0
$$
for every $\omega\in(0,1)$ if
\begin{equation}
y>2(1+\sqrt{2})
\qquad\text{and}\qquad
p\geq\bar p(y).
\end{equation}
This proves the sufficiency of condition \text{(ii)}.\\

\medskip

Finally, we consider $y>2(1+\sqrt{2})$ and $p<\bar p(y).$ By \eqref{eq:A_pbar}, $A(y,p)<0$ and by \eqref{Eq:phat_condition} and \eqref{eq:pbar_order}, $B(y,p)<0$.
This implies that
\begin{equation}
A(y,p)-\omega B(y,p)\geq0 \qquad
\iff \qquad
\omega\geq\frac{A(y,p)}{B(y,p)}.
\end{equation}
Substituting the expressions for $A(y,p)$ and $B(y,p)$ gives a threshold for $\omega$, denoted by $\bar \omega(y,p)$:
\begin{equation} \label{Eq:homogeneouss_omega_threshold}
\omega
\geq
\frac{
4+4y(1+p)+y^2(-1+3p+p^2)
}{
p(3y+2)\left(2+y(-1+2p)\right)
}
=
\bar\omega(y,p).
\end{equation}
Hence, when
$$
y>2(1+\sqrt{2}),
\qquad
p<\bar p(y),
$$
firms participate if and only if
$$
\omega\geq\bar\omega(y,p).
$$
This proves that condition \text{(iii)} is sufficient for participation.\\

\medskip

These three conditions partition the parameter space. Moreover, condition \text{(iii)} is both necessary and sufficient when $y>2(1+\sqrt{2})$ and $p<\bar p(y)$.
Therefore, if none of the three conditions is satisfied, then
$$
y>2(1+\sqrt{2}),
\qquad
p<\bar p(y),
\qquad
\omega<\bar\omega(y,p).
$$
Thus $\Delta_i<0$, so firms do not participate. This proves the necessity part of the proposition.
\end{proof}

\begin{proof} [Proof of Proposition \ref{prop:cost_bottleneck}]
Suppose $p_1=p_2=p$. By Lemma~\ref{lem:symmetric_information}, we consider
the interior representation of the participation gain for both firms. Let

$$
c:=c_1,\qquad
z:=\frac{c_2}{c_1}>1,\qquad
y:=\frac{b_A}{c_1}>0.
$$

Substituting $c_1=c$, $c_2=cz$, $b_A=cy$, $p_1=p_2=p$ into participation gains, 
by Lemma \ref{lem:denominator_sign}, the signs of $\Delta_1$ and $\Delta_2$ are determined by their respective numerators. Let $Num(\Delta_i)$ be the numerator of the participation gain function of Firm $i$.

After rearranging the terms in the numerator, we have
\begin{align}
\operatorname{Num}(\Delta_1)
&=
D^2p\omega(1-p)(1-\omega)^2(1+y)
G_1(z,y,p,\omega),
\label{eq:num_delta_1}\\
\operatorname{Num}(\Delta_2)
&=
D^2p\omega(1-p)(1-\omega)^2(z+y)
\,G_2(z,y,p,\omega),
\label{eq:num_delta_2}
\end{align}
where $(1+y) G_1(z,y,p,\omega)$ and $(z+y) G_2(z,y,p,\omega)$ denote the firm-specific polynomial factors in
$z,y,p,\omega$. $G_1(z,y,p,\omega)$ and $G_2(z,y,p,\omega)$ are the following polynomials:
\begin{equation}\label{eq:G1_ezplicit}
\begin{aligned}
G_1(z,y,p,\omega)
={}&
-p^4y
\left(
4\omega z+\omega y^2+4\omega y+2zy+y^2
\right)
\\
&\quad\times
\left(
16\omega^2z^2y
+16\omega^2z^2
+20\omega^2zy^2
+24\omega^2zy
+6\omega^2y^3
+8\omega^2y^2
\right.
\\
&\hspace{35mm}\left.
+8\omega zy^2
+4\omega zy
+5\omega y^3
+4\omega y^2
-2zy^2
-y^3
\right)
\\[1mm]
&+p^3
\Big(
-64\omega^3z^4y^2
-128\omega^3z^4y
-64\omega^3z^4
-160\omega^3z^3y^3
\\
&\qquad
-224\omega^3z^3y^2
-64\omega^3z^3y
-116\omega^3z^2y^4
-32\omega^3z^2y^3
+112\omega^3z^2y^2
\\
&\qquad
-20\omega^3zy^5
+104\omega^3zy^4
+160\omega^3zy^3
+3\omega^3y^6
+40\omega^3y^5
+48\omega^3y^4
\\
&\qquad
+32\omega^2z^3y^3
+160\omega^2z^3y^2
+128\omega^2z^3y
+88\omega^2z^2y^4
+480\omega^2z^2y^3
\\
&\qquad
+368\omega^2z^2y^2
+96\omega^2zy^5
+464\omega^2zy^4
+352\omega^2zy^3
+33\omega^2y^6
\\
&\qquad
+144\omega^2y^5
+112\omega^2y^4
+64\omega z^3y^3
+64\omega z^3y^2
+160\omega z^2y^4
\\
&\qquad
+160\omega z^2y^3
+16\omega z^2y^2
+120\omega zy^5
+136\omega zy^4
+32\omega zy^3
\\
&\qquad
+29\omega y^6
+40\omega y^5
+16\omega y^4
-4z^2y^4
-4zy^5
-y^6
\Big)
\\[1mm]
&+p^2
\Big(
192\omega^2z^4y^2
+384\omega^2z^4y
+192\omega^2z^4
+480\omega^2z^3y^3
\\
&\qquad
+736\omega^2z^3y^2
+256\omega^2z^3y
+364\omega^2z^2y^4
+192\omega^2z^2y^3
-208\omega^2z^2y^2
\\
&\qquad
+64\omega^2zy^5
-296\omega^2zy^4
-416\omega^2zy^3
-13\omega^2y^6
-136\omega^2y^5
-144\omega^2y^4
\\
&\qquad
-64\omega z^3y^3
-128\omega z^3y^2
-64\omega z^3y
-176\omega z^2y^4
-416\omega z^2y^3
\\
&\qquad
-224\omega z^2y^2
-164\omega zy^5
-432\omega zy^4
-256\omega zy^3
-50\omega y^6
\\
&\qquad
-144\omega y^5
-96\omega y^4
-32z^3y^3
-32z^3y^2
-68z^2y^4
-64z^2y^3
\\
&\qquad
-44zy^5
-40zy^4
-9y^6
-8y^5
\Big)
\\[1mm]
&-4p(z+y)
\Big(
48\omega z^3y^2
+96\omega z^3y
+48\omega z^3
+72\omega z^2y^3
\\
&\qquad
+104\omega z^2y^2
+32\omega z^2y
+19\omega zy^4
-24\omega zy^3
-48\omega zy^2
\\
&\qquad
-4\omega y^5
-32\omega y^4
-32\omega y^3
-8z^2y^3
-8z^2y^2
-11zy^4
\\
&\qquad
-16zy^3
-4zy^2
-4y^5
-8y^4
-4y^3
\Big)
\\[1mm]
&+4(z+y)^2
\left(
16z^2y^2
+32z^2y
+16z^2
+8zy^3
+8zy^2
-y^4
-8y^3
-8y^2
\right)
\end{aligned}
\end{equation}

and
\begin{align}
G_2(z,y,p,\omega)
={}&
-p^4y
\left(
4\omega zy
+4\omega z
+\omega y^2
+y^2
+2y
\right)
\nonumber\\
&\quad\times
\left(
8\omega^2zy^2
+24\omega^2zy
+16\omega^2z
+6\omega^2y^3
+20\omega^2y^2
+16\omega^2y
\right.
\nonumber\\
&\hspace{35mm}\left.
+4\omega zy^2
+4\omega zy
+5\omega y^3
+8\omega y^2
-y^3
-2y^2
\right)
\nonumber\\[1mm]
&+p^3
\Big(
48\omega^3z^2y^4
+160\omega^3z^2y^3
+112\omega^3z^2y^2
-64\omega^3z^2y
-64\omega^3z^2
\nonumber\\
&\qquad
+40\omega^3zy^5
+104\omega^3zy^4
-32\omega^3zy^3
-224\omega^3zy^2
-128\omega^3zy
\nonumber\\
&\qquad
+3\omega^3y^6
-20\omega^3y^5
-116\omega^3y^4
-160\omega^3y^3
-64\omega^3y^2
\nonumber\\
&\qquad
+112\omega^2z^2y^4
+352\omega^2z^2y^3
+368\omega^2z^2y^2
+128\omega^2z^2y
\nonumber\\
&\qquad
+144\omega^2zy^5
+464\omega^2zy^4
+480\omega^2zy^3
+160\omega^2zy^2
+33\omega^2y^6
\nonumber\\
&\qquad
+96\omega^2y^5
+88\omega^2y^4
+32\omega^2y^3
+16\omega z^2y^4
+32\omega z^2y^3
\nonumber\\
&\qquad
+16\omega z^2y^2
+40\omega zy^5
+136\omega zy^4
+160\omega zy^3
+64\omega zy^2
\nonumber\\
&\qquad
+29\omega y^6
+120\omega y^5
+160\omega y^4
+64\omega y^3
-y^6
-4y^5
-4y^4
\Big)
\nonumber\\[1mm]
&+p^2
\Big(
-144\omega^2z^2y^4
-416\omega^2z^2y^3
-208\omega^2z^2y^2
+256\omega^2z^2y
+192\omega^2z^2
\nonumber\\
&\qquad
-136\omega^2zy^5
-296\omega^2zy^4
+192\omega^2zy^3
+736\omega^2zy^2
+384\omega^2zy
\nonumber\\
&\qquad
-13\omega^2y^6
+64\omega^2y^5
+364\omega^2y^4
+480\omega^2y^3
+192\omega^2y^2
\nonumber\\
&\qquad
-96\omega z^2y^4
-256\omega z^2y^3
-224\omega z^2y^2
-64\omega z^2y
\nonumber\\
&\qquad
-144\omega zy^5
-432\omega zy^4
-416\omega zy^3
-128\omega zy^2
-50\omega y^6
\nonumber\\
&\qquad
-164\omega y^5
-176\omega y^4
-64\omega y^3
-8zy^5
-40zy^4
-64zy^3
\nonumber\\
&\qquad
-32zy^2
-9y^6
-44y^5
-68y^4
-32y^3
\Big)
\nonumber\\[1mm]
&+4p(1+y)
\Big(
32\omega z^2y^3
+48\omega z^2y^2
-32\omega z^2y
-48\omega z^2
\nonumber\\
&\qquad
+32\omega zy^4
+24\omega zy^3
-104\omega zy^2
-96\omega zy
+4\omega y^5
\nonumber\\
&\qquad
-19\omega y^4
-72\omega y^3
-48\omega y^2
+4z^2y^3
+4z^2y^2
+8zy^4
\nonumber\\
&\qquad
+16zy^3
+8zy^2
+4y^5
+11y^4
+8y^3
\Big)
\nonumber\\[1mm]
&-4(1+y)^2
\left(
8z^2y^2
-16z^2
+8zy^3
-8zy^2
-32zy
+y^4
-8y^3
-16y^2
\right).
\label{eq:G2_ezplicit}
\end{align}

Since $1+z>0$ and $z+y>0$, 

\begin{equation} \label{Eq:G_means_Delta}
    \text{sign}(\Delta_i) = \text{sign}(G_i) \qquad \text{for} \quad i=1,2.
\end{equation}
Next, it is sufficient to show that whenever $z>1$, $y>0$, $0<p,<1$, $0<\omega<1$,

$$
G_1(z,y,p,\omega)>G_2(z,y,p,\omega)
$$

When we factor $G_1(z,y,p,\omega)-G_2(z,y,p,\omega)$, we have 

\begin{equation}
G_1(z,y,p,\omega)-G_2(z,y,p,\omega)
=
(z-1)H(z,y,p,\omega).
\label{eq:G_difference_factor}
\end{equation}

where $H(z,y,p,\omega)$ is a polynomial of $z,y,p,\omega$.
Because $H(z,y,p,\omega)$ is lengthy, we omit the exact expression. However, collecting it by using Mathematica in terms of
\begin{equation}
z-1,\qquad y,\qquad p,\qquad 1-p,\qquad
\omega,\qquad 1-\omega.
\end{equation}
 gives
\begin{equation}
H(z,y,p,\omega)
=
4
\sum_{\nu=1}^{252}
\gamma_{\nu}\,
(z-1)^{t_{\nu}}
y^{\ell_{\nu}}
p^{m_{\nu}}
(1-p)^{4-m_{\nu}}
\omega^{n_{\nu}}
(1-\omega)^{3-n_{\nu}},
\label{eq:H_positive_representation}
\end{equation}
where $t_{\nu},\ell_{\nu}\geq0$, $0\leq m_{\nu}\leq4$,
 $0\leq n_{\nu}\leq3$, and $\gamma_{\nu}>0$ for all $\nu=1,2,..,252$. Therefore,
$$
H(z,y,p,\omega)>0.
$$
Since $z-1>0$,
\begin{equation} \label{Eq:Firm1_better}
G_1(z,y,p,\omega)>G_2(z,y,p,\omega).
\end{equation}

Now suppose that Firm 2 weakly benefits
from participation:
$$
\Delta_2\geq0.
$$
Then by \eqref{Eq:G_means_Delta},
$$
G_2\geq0.
$$
Then \eqref{Eq:Firm1_better} implies that, 
$$
G_1>0.
$$
Again, by \eqref{Eq:G_means_Delta},
$$
\Delta_1>0.
$$
Therefore,
\begin{equation}
    \Delta_2\geq0
\quad\Longrightarrow\quad
\Delta_1>0.
\end{equation}

\end{proof}

\begin{proof}[Proof of Proposition \ref{prop:infopower_bottleneck}]
Because $c_1=c_2=c$, platform formation makes the firms symmetric in the
quantity decisions. Hence, $\Pi_1([1,1])=\Pi_2([1,1])$. Therefore,
$$ \Delta_1=\Pi_1([1,1])-\Pi_1([0,0]),
\qquad
\Delta_2=\Pi_1([1,1])-\Pi_2([0,0]),
$$
which implies
\begin{equation}
\Delta_2-\Delta_1
=
\Pi_1([0,0])-\Pi_2([0,0]).
\label{eq:info-bottleneck-difference}
\end{equation}
Thus, it is sufficient to prove the stronger result:
\begin{equation}\label{p_dominant_shares_more}
\Pi_1([0,0])>\Pi_2([0,0]) \qquad \text{if} \qquad p_1>p_2.
\end{equation}
 Define the following transformations:
$$
y:=\frac{b_A}{c}>0,
\qquad
x:=\frac{D}{a_A}>0.
$$
We show the condition \eqref{p_dominant_shares_more} for interior and boundary equilibria under
$\vec V=[0,0]$ separately.

\medskip
\noindent
\textit{i) Interior equilibrium.}

When the private-information equilibrium is interior, all
quantities are strictly positive. By \eqref{Eq:FOC1},
$$a_A + D
- b_A\!\left(
p_j q^+_{j}
+ (1-p_j) q^-_{j}
\right)= 2(b_A+c_i) q^+_{i}$$
and by \eqref{Eq:FOC2}, 
$$a_A
+\frac{\omega(1-p_i)}{1-\omega p_i}D
-\frac{1-\omega p_i-\omega p_j+\omega p_ip_j}
      {1-\omega p_i}\,
b_A q_j^-
-\frac{\omega(1-p_i)p_j}{1-\omega p_i}\,
b_A q_j^+
=
2(b_A+c_i)q_i^-.$$

By substituting the left term with right, Firm $i$'s conditional equilibrium profit given its type $t$ can be written as
$(b_A+c)q_i([0,0]|t)^2$ for $i\in\{1,2\}$ and $t\in\{-,(+,A)\}$. Hence,

$$
\Pi_i([0,0])
=
(b_A+c)
\left(
(1-\omega p_i)(q_i^-)^2
+
\omega p_i(q_i^+)^2
\right).
$$

Substituting the interior equilibrium quantities and rearranging gives
\begin{equation}
\Pi_1([0,0])-\Pi_2([0,0])
=
\frac{D^2}{c}
\frac{
\omega(p_1-p_2)(1-\omega)^2(1+y)
}{
4(1+y)^2(1-p_1\omega)(1-p_2\omega)
-y^2p_1p_2(1-\omega)^2
},
\label{eq:interior-profit-ordering}
\end{equation}

Since $0<p_2<p_1<1$ and $0<\omega<1$,
$$
\begin{aligned}
1-p_i\omega>1-\omega,
\qquad i=1,2,& \\
\Rightarrow 4(1+y)^2(1-p_1\omega)(1-p_2\omega)
-y^2p_1p_2(1-\omega)^2
>&
(1-\omega)^2
\left[
4(1+y)^2-y^2
\right]
>0.
\end{aligned}
$$
Similarly, since $0<p_2<p_1<1$, $0<\omega<1$ and $y>0$, the numerator of \eqref{eq:interior-profit-ordering} is strictly positive. Thus,
\begin{equation}\label{eq:interior-private-profit-order} 
\Pi_1([0,0])-\Pi_2([0,0])>0\qquad\text{if }\qquad p_1>p_2 \text{ and } \mathbf{q}^*=\mathbf{q}^I.
\end{equation}

\medskip
\noindent
\textit{ii) Boundary equilibrium.}

Since, $q_1^-=0, q_1^+>0,q_2^->0, q_2^+>0$ when the equilibrium is the boundary solution, equilibrium-profit representations at boundary case are
\begin{align}
\Pi_1([0,0])
&=
(b_A+c)\omega p_1(q_1^+)^2,
\label{Eq:boundaryfirm1profit}\\
\Pi_2([0,0])
&=
(b_A+c)
\left(
(1-\omega p_2)(q_2^-)^2
+\omega p_2(q_2^+)^2
\right).\label{Eq:boundaryfirm2profit}
\end{align}.

Substituting the boundary equilibrium quantities into expressions \eqref{Eq:boundaryfirm1profit} and \eqref{Eq:boundaryfirm2profit} gives

\begin{equation} \label{eq:boundary-profit-difference-expanded}
\begin{aligned}
\Pi_1([0,0])-\Pi_2([0,0])
=&
\frac{a_A^2}{c}
\frac{1}{
4(1+y)
\left[
4(1+y)^2(1-p_2\omega)
-y^2p_1(p_2+\omega-2p_2\omega)
\right]
}
\\
&\times
\Bigg(
-\Big(
4(1+y)^2(1-p_1\omega)(1-p_2\omega)
-y^2p_1p_2(1-\omega)^2
\Big)
\\
&\qquad
-8\omega(1-p_1)(1-p_2\omega)(1+y)^2x
\\
&\qquad
-4\omega(1+y)^2
\Big(
p_1p_2\omega-p_1-2p_2\omega+p_2+\omega
\Big)x^2
\Bigg).
\end{aligned}   
\end{equation}
Since $0<p_2<p_1<1$ and $0<\omega<1$,
$$
\begin{aligned}
p_2+\omega-2p_2\omega
<
1-p_2\omega,& \\
\Rightarrow 4(1+y)^2(1-p_2\omega)
-y^2p_1(p_2+\omega-2p_2\omega)
>&
(1-p_2\omega)
\left(
4(1+y)^2-y^2p_1
\right)
>0.
\end{aligned}
$$

Moreover, since $c>0$, the denominator of \eqref{eq:boundary-profit-difference-expanded} is strictly positive. Thus, the sign of
$\Pi_1([0,0])-\Pi_2([0,0])$ is determined by the numerator of
\eqref{eq:boundary-profit-difference-expanded}. In the numerator, $a_A>0$, hence, it is sufficient to show that the remaining term is positive.

However, this term does not need to be
positive for every $x>0$. It suffices to show its positivity over
the parameter region in which the boundary candidate is the
equilibrium.

By \eqref{eq:complete-equilibrium}, when $p_1>p_2$, the boundary
solution is the equilibrium only if $$\phi_D>0 \qquad \text{and} \qquad
x=\frac{D}{a_A}\geq
\frac{\phi_a}{\phi_D}.
$$
Therefore, showing the quadratic expression
in \eqref{eq:boundary-profit-difference-expanded} is strictly positive
for every $x\geq \frac{\phi_a}{\phi_D}$ is sufficient. For $c_1=c_2=c$ and $b_A=cy$, we substitute
$\kappa_i=2(b_A+c)$ and
$r_i=\frac{\omega(1-p_i)}{1-\omega p_i}$
into equations \eqref{eq:coefficient-a} and \eqref{eq:coefficient-D}:

\begin{align}
\phi_a
&=
\frac{c^3(y+2)}
{(1-\omega p_1)(1-\omega p_2)}
\Big(
4(1+y)^2(1-\omega p_1)(1-\omega p_2)
-y^2p_1p_2(1-\omega)^2
\Big),
\label{eq:phi-a-symmetric-cost}
\\
\phi_D
&=
\frac{c^3(y+2)}
{(1-\omega p_1)(1-\omega p_2)}
\Bigg(
-2\omega(1+y)
\Big(
2(1-p_1)(1-p_2\omega)
\nonumber\\
&\hspace{42mm}
+y\big(
p_1p_2\omega+p_1p_2+p_1\omega
-3p_1-2p_2\omega+2
\big)
\Big)
\Bigg).
\label{eq:phi-D-symmetric-cost}
\end{align}

The common factor $\frac{c^3(y+2)}
{(1-\omega p_1)(1-\omega p_2)} $ is strictly positive. Hence, it cancels from the ratio $\phi_a/\phi_D$. Thus, the boundary equilibrium conditions are 

\begin{equation}\label{eq:x_threshold}
    x\geq x^*=\frac{
4(1+y)^2(1-p_1\omega)(1-p_2\omega)
-y^2p_1p_2(1-\omega)^2
}{
-2\omega(1+y)
\left(
2(1-p_1)(1-p_2\omega)
+y(
p_1p_2\omega+p_1p_2+p_1\omega
-3p_1-2p_2\omega+2
)
\right)
}
\end{equation}
and
\begin{equation}\label{eq:phiDconditon}
-2\omega(1+y)
\Big(
2(1-p_1)(1-p_2\omega)
+y\big(
p_1p_2\omega+p_1p_2+p_1\omega
-3p_1-2p_2\omega+2
\big)
\Big)
>0,
\end{equation}
where $x^*$ denotes a threshold for $x$.
Since $2\omega(1+y)>0$, condition \eqref{eq:phiDconditon} is equivalent to
\begin{equation}
2(1-p_1)(1-p_2\omega)
+y\big(
p_1p_2\omega+p_1p_2+p_1\omega
-3p_1-2p_2\omega+2
\big)
<0.
\label{eq:phiD-condition-normalized}
\end{equation}

Because $2(1-p_1)(1-p_2\omega)>0$ and $y>0$, it follows that
\begin{equation}
p_1p_2\omega+p_1p_2+p_1\omega
-3p_1-2p_2\omega+2<0.
\label{eq:boundary-bracket-negative}
\end{equation}

Note that, 
\begin{align*}
&
\Big(
p_1p_2\omega+p_1p_2+p_1\omega
-3p_1-2p_2\omega+2
\Big)-
\Big(
p_1p_2\omega-p_1-2p_2\omega+p_2+\omega
\Big)
\\
&=(1-p_1)(2-p_2-\omega)>0.\\
\Rightarrow &0>\Big(p_1p_2\omega+p_1p_2+p_1\omega
-3p_1-2p_2\omega+2\Big)>\Big(p_1p_2\omega-p_1-2p_2\omega+p_2+\omega\Big).
\end{align*}
Hence, the coefficient of $x^2$ in \eqref{eq:boundary-profit-difference-expanded} is strictly positive, implying that the quadratic expression determining
$\Pi_1([0,0])-\Pi_2([0,0])$ is strictly convex in $x$.

It remains to show that this quadratic expression is positive throughout the
boundary-equilibrium region. Thus, the boundary equilibrium requires $x\geq x^*$.

Since $1-p_i\omega>1-\omega$ for $i=1,2$,
\begin{equation} \label{eq:boundary-threshold-numerator-positive}
4(1+y)^2(1-p_1\omega)(1-p_2\omega)
-y^2p_1p_2(1-\omega)^2>
(1-\omega)^2
\left(4(1+y)^2-y^2p_1p_2\right)
>0.
\end{equation}
Therefore, the quadratic expression in
\eqref{eq:boundary-profit-difference-expanded} is strictly negative at $x=0$.
Together with the expression is strictly convex, it implies that the expression has
one negative root and one positive root. Moreover, by \eqref{eq:phiDconditon}, the denominator defining $x^*$ is
strictly positive, while
\eqref{eq:boundary-threshold-numerator-positive} shows that its numerator
is strictly positive. Hence, $x^*>0$.

The last step to show is
that the unique positive root lies below $x^*$. Since $x^*$ is the
left endpoint of the boundary-equilibrium region, we evaluate the
quadratic expression at $x=x^*$ which is
\begin{equation} \label{eq:boundary-quadratic-at-threshold}
\begin{aligned}
&
\left(
\frac{
(p_1-p_2)(1-\omega)^2
}{
\omega
\left(
2(1-p_1)(1-p_2\omega)
+y\left(
p_1p_2\omega+p_1p_2+p_1\omega
-3p_1-2p_2\omega+2
\right)
\right)^2
}
\right)
\\
&\qquad\times
\left(
4(1+y)^2(1-p_2\omega)
-y^2p_1(p_2+\omega-2p_2\omega)
\right)
\\
&\qquad\times
\left(
4(1+y)^2(1-p_1\omega)(1-p_2\omega)
-y^2p_1p_2(1-\omega)^2
\right).
\end{aligned}
\end{equation}

Because $p_1-p_2>0$, $0<\omega<1$ the first factor in \eqref{eq:boundary-quadratic-at-threshold} is strictly
positive, its
denominator is strictly positive by
\eqref{eq:phiD-condition-normalized}. The second factor is strictly
positive by
\eqref{eq:boundary-profit-difference-expanded}, and the third factor is
strictly positive by
\eqref{eq:boundary-threshold-numerator-positive}. Hence, the quadratic
expression is strictly positive at $x=x^*$.

Thus, the unique positive root must be strictly below $x^*$. Consequently, the quadratic expression is strictly positive for every $x\geq x^*$, and hence
\begin{equation}\label{eq:boundary-private-profit-order}
\Pi_1([0,0])-\Pi_2([0,0])>0 \qquad\text{if } \qquad p_1>p_2 \text{ and } \mathbf{q}^*=\mathbf{q}^C.
\end{equation}

Combining \eqref{eq:interior-private-profit-order} and \eqref{eq:boundary-private-profit-order}, we have
$$
\Pi_1([0,0])-\Pi_2([0,0])>0
\qquad \iff \qquad p_1>p_2.
$$
By \eqref{eq:info-bottleneck-difference},
$$
\Delta_2-\Delta_1
=
\Pi_1([0,0])-\Pi_2([0,0])>0.
$$
Consequently, if $\Delta_1\geq0$ then $\Delta_2>0$.
\end{proof}

\subsection{Construction of Figure \ref{fig:asymmetric_cost_firm2}} \label{App:inefficientplot}

Define the following transformations:
$$
y:=\frac{b_A}{c_2}>0, \qquad z:=\frac{c_1}{c_2}<1.
$$

By \ref{App:participation_gain_derivation}, factoring the terms of numerator $f(\cdot)$ with respect to $z$ gives

\begin{equation}
    \Delta_2\geq0 \iff \beta_0+\beta_1 z+\beta_2 z^2+\beta_3 z^3+\beta_4 z^4+\beta_5 z^5+\beta_6 z^6+\beta_7 z^7\geq0.
\end{equation}
with coefficients $\beta_0$-$\beta_7$. Then, define the cumulative summation denoted by 

$$S_t:=\sum^{t}_{k=0} \beta_k, \qquad \forall t\in \{0,1,2,3,4,5,6,7\}$$

By using a Computer Algebra System (CAS), Mathematica, we find the conditions for $0<p<1$, $0<\omega<1$, and selected $y$ values for which the coefficients
\begin{equation} \label{eq:csum}
S_t \geq0, \qquad \forall t\in \{0,1,2,3,4,5,6,7\}.
\end{equation}
Since $z<1$, \eqref{eq:csum}, implies $\Delta_2\geq0$.

\end{document}